\documentclass[12pt]{article}
\usepackage{amssymb,latexsym,amsmath}
\usepackage{amsthm}
\usepackage{geometry}
\usepackage{graphicx,color}
\usepackage{lineno}
\newtheorem{thm}{Theorem}

\theoremstyle{plain}

\theoremstyle{definition}

\newtheorem{method}{Method}
\numberwithin {equation}{section}
\usepackage{longtable}

\usepackage{float}
\restylefloat{table}

\begin{document}
\title
{A new class of optimal four-point methods with \\ convergence order 16 for solving nonlinear equations}
\author{Somayeh Sharifi$^a$\thanks{s.sharifi@iauh.ac.ir} \and Mehdi Salimi$^b$\thanks{Corresponding author: mehdi.salimi@tu-dresden.de}
 \and Stefan
Siegmund$^b$\thanks{stefan.siegmund@tu-dresden.de}\and Taher
Lotfi$^c$\thanks{lotfi@iauh.ac.ir}}
\date{}
\maketitle
\begin{center}
$^{a}$Young Researchers and Elite Club, Hamedan Branch, Islamic
Azad University, Hamedan, Iran\\
$^{b}$Center for Dynamics, Department of Mathematics, Technische
Universit{\"a}t Dresden, 01062 Dresden, Germany\\
$^{c}$Department of Mathematics, Hamedan Branch, Islamic Azad
University, Hamedan, Iran\\
\end{center}
\maketitle

\begin{abstract}\noindent
We introduce a new class of optimal iterative methods without memory for approximating a simple root of a given
nonlinear equation. The proposed class uses four function evaluations and one first derivative evaluation per iteration and it is therefore optimal in the sense of Kung and Traub's conjecture. We present the construction, convergence
analysis and numerical implementations, as well as comparisons of accuracy and basins of attraction between our method and existing optimal methods for several test problems.

\bigskip\noindent
\textbf{Keywords}: Simple root, four-step iterative method, Kung and Traub conjecture, optimal order of
convergence, computational efficiency.
 \end{abstract}

\section{Introduction}
Solving nonlinear equations is a basic and extremely valuable tool
in all fields in the sciences and in engineering. One can
distinguish between two general approaches for solving nonlinear
equations numerically, namely, one-step and multi-step methods.
Multi-step methods overcome some computational issues encountered
with one-step iterative methods, typically they allow us to
achieve a greater accuracy with the same number of function
evaluations. In this context an unproved conjecture by Kung and
Traub \cite{Kung} plays a central role, it states that an optimal
multi-step method without memory which uses $n+1$ evaluations
could achieve a convergence order of $2^{n}$. Considering this
conjecture, many optimal two-step and three-step methods have been
presented. However, because of complexity in construction and
development, optimal four-point methods are rare and can be
considered as an active research problem.

Prominent optimal two-point methods have been introduced by
Jarratt \cite{Jarrat}, King \cite{King} and Ostrowski
\cite{Ostrowski}. Some optimal three-point methods have been
proposed by Chun and Lee \cite{Chun1}, Cordero et al.\
\cite{Cordero5}-\cite{Cordero22}, Khattri and Steihaug
\cite{Khattri}, Lotfi et al.\ \cite{Lotfi0}-\cite{Lotfi1},
Petkovic et al. \cite{Petkovic1,Petkovic2} and Sharma et al.\
\cite{Sharma1}. Neta \cite{Neta0} has presented methods with
convergence orders $8$ and $16$. Babajee and Thukral
\cite{Babajee} developed a four-point method with convergence
order $16$ based on the King family of methods. In
\cite{Geum1}-\cite{Geum3} Geum and Kim provided three methods with
convergence order 16 by using weight function methods.

We construct a new optimal class of four-point methods without memory which uses
five function evaluations per iteration. The paper is organized as follows: Section 2 is devoted to the construction of the new
optimal class. We first construct classes of optimal two-point and three-point methods and then utilize them in the three first steps of our new method. The section also includes convergence analysis of all these
methods. Numerical performance and comparisons with other methods are
illustrated in Section 3. A conclusion is provided in Section 4.

\section{Main results: \\
Construction, error and convergence analysis} This section deals
with construction, and error and convergence analysis of our
method. First, we try to introduce an optimal two-point  class
(this class contains no originality, and we review it here only
for easy-reference in constructing the next two classes), then an
optimal three-point class is presented, and, finally, our optimal
four-point class will be developed.

\subsection{Construction of an optimal two-point class}
In this section we construct a new optimal two-point method for
solving nonlinear equations which employs Newton's one-point
method \cite{Ostrowski,Traub} and suitable weight functions for
evaluations at two points.

Newton's method \cite{Ostrowski,Traub}

\begin{equation}\label{a1}
  \begin{split}
  y_{n}&=x_n-\dfrac{f(x_n)}{f'(x_n)},
    \quad (n \in \mathbb{N}_0 = \{0, 1, \dots\})
  ,
  \end{split}
\end{equation}
where $x_0$ denotes the initial approximation of $x^{*}$, is of
convergence order two. To increase the order of convergence, we
add one Newton step to the method \eqref{a1} to get
\begin{equation}\label{a2}
\begin{cases}
y_n=x_n-\dfrac{f(x_n)}{f'(x_n)},\\
x_{n+1}=y_n-\dfrac{f(y_n)}{f'(y_n)}.
\end{cases}
\end{equation}
The method \eqref{a2} uses four function evaluations to achieve
order four, therefore the method is not an optimal two-point
method. We modify \eqref{a2} by approximating $f'(y_n)$ by
\begin{equation*}
  f'(y_n) \approx \frac{f'(x_n)}{G(t_n)},
\end{equation*}
using only the values $f(x_n)$, $f(y_n)$, and $f'(x_n)$. More
precisely, we use the abbreviations $t_n=\frac{f(y_n)}{f(x_n)}$
and utilize Mathematica \cite{Hazrat} to carefully choose the
weight function $G : \mathbb{R} \rightarrow \mathbb{R}$ from a
class of admissible functions such that the scheme
\begin{equation}\label{a3}
\begin{cases}
y_n=x_n-\dfrac{f(x_n)}{f'(x_n)},\\
x_{n+1}=y_n- G(t_n) \cdot\dfrac{f(y_n)}{f'(x_n)},
\end{cases}
\end{equation}
is of order four.

\begin{thm}\label{theorem1}
Assume that $ f: D \subset \mathbb{R} \rightarrow
\mathbb{R}$ is four times continuously differentiable
and has a simple zero $x^{*}\in D$ and
$G\in C^3\mathbb{R}$ is
sufficiently  continuously differentiable.
If
\begin{equation*}
  \begin{split}
    G_{0} &=1, \quad G_{1}=2,\quad |G_2|<\infty,
    \end{split}
\end{equation*}
where $G_{i}=\frac{d^i G(t_n)}{dt_n^i}|_{0}$ for $i=0,1,\ldots$,
and the initial point $x_{0}$ is sufficiently close to $x^{*}$,
then the sequence $\{x_n\}$ defined by \eqref{a3} converges to
$x^{*}$ and the order of convergence is four.
\end{thm}

\begin{proof}
Let $e_{n}:=x_{n}-x^{*}$, $e_{n,y}:=y_{n}-x^{*}$,
 and
$c_{n}:=\frac{f^{(n)}(x^{*})}{n!f^{'}(x^{*})}$ for $n \in \mathbb{N}$. Using the fact that $f(x^{*})=0$, the Taylor expansion of $f$ at $x^*$ yields
\begin{equation}\label{a4}
f(x_{n}) = f^{'}(x^*)(e_{n} +
c_{2}e_{n}^{2}+c_{3}e_{n}^{3}+c_{4}e_{n}^{4})+O(e_{n}^{5}),
\end{equation}
and expanding $f'$ at $x^*$ we get
\begin{equation}\label{a5}
f^{'}(x_{n})=f^{'}(x^{*})(1+2c_{2}e_{n}+3c_{3}e_{n}^{2}+4c_{4}e_{n}^{3}+5c_{5}e_{n}^{4})+O(e_{n}^{5}).
\end{equation}
Therefore
\begin{equation*}\label{a6}
\dfrac{f(x_{n})}{f'(x_{n})}=e_{n}-c_{2}e_{n}^{2}+2(c_{2}^{2}-c_{3})e_{n}^{3}+(-4c_2^3+7c_2c_3-3c_4)e_n^4+O(e_n^5),
\end{equation*}
and hence
\begin{equation*}\label{a7}
e_{n,y}= y_n-x^*=c_{2}e_{n}^{2}+ O(e_n^3).
\end{equation*}
\\
For $f(y_n)$, we also have
\begin{equation}\label{a8}
f(y_{n})=f^{'}(x^{*})\left(e_{n,y} +
c_{2}e_{n,y}^{2}+c_{3}e_{n,y}^{3}+c_4e_{n,y}^{4}\right)+O(e_{n,y}^{5}).
\end{equation}\\
By (\ref{a4}) and (\ref{a8}), we obtain

\begin{equation}\label{a9}
t_n=\dfrac{f(y_n)}{f(x_n)}\\
=c_2 e_n+(-3c_2^2+2c_3)e_n^2+(8c_2^3-10c_2c_3+3c_4)e_n^3+O(e_n^4),
\end{equation}

and expanding $G$ at $0$ yields

\begin{equation}\label{a10}
G(t_n)=G_{0}+G_{1}t_n+\frac{1}{2}G_{2}t_n^2+O(t_n^3).
\end{equation}

By substituting (\ref{a4})-(\ref{a10})
 into (\ref{a3}), we obtain
\begin{equation*}
e_{n+1}=x_{n+1}-x^*=R_2e_n^2+R_3e_n^3+R_4e_n^4+O(e_n^5),
\end{equation*}
where
\begin{equation*}
\begin{split}
R_2&=-c_2(-1+G_{0}),\\
R_3&=-c_2^2(-2+G_{1}),\\
R_4&=-c_2c_3+c_2^3(5-\frac{1}{2}G_2).
\end{split}
\end{equation*}
\\
In general $R_{4}\neq 0$, however, by setting $R_2 = R_3 = 0$, the
convergence order becomes four. Sufficient conditions are given by
the following set of equations
\begin{equation*}
  \begin{array}{lrl}
    G_{0} =1 & \quad \Rightarrow & R_2 = 0 ,
  \\[1ex]
    G_{1}=2 & \Rightarrow & R_3=0 ,
    \\[1ex]
   |G_{2}|<\infty & \Rightarrow & R_4\neq0 ,
  \end{array}
\end{equation*}
and the error equation
becomes
\begin{equation*}
  e_{n+1}=R_{4}e_n^4+O(e_n^5),
\end{equation*}
which finishes the proof.
\end{proof}
\subsection{Construction of an optimal three-point class}
To increase the order of convergence, we add one Newton step to
the method \eqref{a3} to get
\begin{equation}\label{b1}
\begin{cases}
y_n=x_n-\dfrac{f(x_n)}{f'(x_n)},\\
z_{n}=y_n-G(t_n)\cdot\dfrac{f(y_n)}{f'(x_n)},\\
x_{n+1}=z_n-\dfrac{f(z_n)}{f'(z_n)}.
\end{cases}
\end{equation}
The method \eqref{b1} uses five function evaluations and therefore the
method is not an optimal three-point method. We modify \eqref{b1}
by approximating $f'(z_n)$ by
\begin{equation*}
  f'(z_n) \approx \frac{f'(x_n)}{H(t_n, s_n, u_n)},
\end{equation*}
using only the values $f(x_n)$, $f(y_n)$, $f(z_n)$ and $f'(x_n)$.
More precisely, we use the abbreviations
$t_n=\frac{f(y_n)}{f(x_n)}$, $s_n=\frac{f(z_n)}{f(y_n)}$,
$u_n=\frac{f(z_n)}{f(x_n)}$ and utilize Mathematica \cite{Hazrat}
to carefully choose the weight functions $H :
\mathbb{R}^3\rightarrow \mathbb{R}$ from a class of admissible
functions such that the scheme
\begin{equation}\label{b2}
\begin{cases}
y_n=x_n-\dfrac{f(x_n)}{f'(x_n)},\\
z_{n}=y_n- G(t_n) \cdot\dfrac{f(y_n)}{f'(x_n)},\\
x_{n+1}=z_n-H(t_n, s_n, u_n)\cdot\dfrac{f(z_n)}{f'(x_n)},
\end{cases}
\end{equation}
is of order eight.

\begin{thm}\label{theorem2}
Assume that $ f: D \subset \mathbb{R} \rightarrow
\mathbb{R}$ is eight times continuously differentiable and has a simple
zero $x^{*}\in D$ and $G\in C^3\mathbb{R}$ and
$H:\mathbb{R}^3\rightarrow \mathbb{R}$ are
sufficiently often differentiable functions. If
\begin{equation*}
  \begin{split}
   & G_{2} =10, \quad G_{3}=-36,\\
   & H_{0, 0, 0}=1, \quad H_{1, 0 ,0}=2, \quad H_{0, 1, 0}=1,\\
   & H_{2,0,0}=12, \quad H_{0,0,1}=4, \quad H_{1,1,0}=0,
    \end{split}
\end{equation*}
where $H_{i,j,k}=\frac{\partial^{i+j+k}H(t_n,s_n,u_n)}{\partial
t_n^i\partial s_n^j\partial u_n^k}|_{(0,0,0)}$ for
$i,j,k=0,1,2,3,\ldots$, and the initial point $x_{0}$ is
sufficiently close to $x^{*}$, then the sequence $\{x_n\}$ defined
by \eqref{b2} converges to $x^{*}$ and the order of
convergence is eight.
\end{thm}

\begin{proof}
Let $e_{n}:=x_{n}-x^{*}$, $e_{n,y}:=y_{n}-x^{*}$,
$e_{n,z}:=z_{n}-x^{*}$
 and
$c_{n}:=\frac{f^{(n)}(x^{*})}{n!f^{'}(x^{*})}$ for $n \in
\mathbb{N}$. Using the fact that $f(x^{*})=0$, the Taylor
expansion of $f$ at $x^*$ yields
\begin{equation}\label{b3}
f(x_{n}) = f^{'}(x^*)(e_{n} +
c_{2}e_{n}^{2}+c_{3}e_{n}^{3}+\ldots+c_{8}e_{n}^{8})+O(e_{n}^{9}),
\end{equation}
and expanding $f'$ at $x^*$ we get
\begin{equation}\label{b4}
f^{'}(x_{n})=f^{'}(x^{*})(1+2c_{2}e_{n}+3c_{3}e_{n}^{2}+\ldots+9c_{9}e_{n}^{8})+O(e_{n}^{9}).
\end{equation}
Therefore
\begin{equation*}\label{c5}
\begin{split}
\dfrac{f(x_{n})}{f'(x_{n})}&=e_{n}-c_{2}e_{n}^{2}+2(c_{2}^{2}-c_{3})e_{n}^{3}+(-4c_2^3+7c_2c_3-3c_4)e_n^4+(8c_2^4-20c_2^2c_3+6c_3^2\\
&+10c_2c_4-4c_5)e_n^5+(-16c_2^5+52c_2^3c_3-28c_2^2c_4+17c_3c_4+c_2(-33c_3^2+13c_5))e_n^6\\
&+2(16c_2^6-64c_2^4c_3-9c_3^3+36c_2^3c_4-46c_2c_3c_4+6c_4^2+9c_2^2(7c_3^2-2c_5)+11c_3c_5)e_n^7\\
&+(-64c_2^7+304c_2^5c_3-176c_2^4c_4+348c_2^2c_3c_4+c_4(-75c_3^2+31c_5)\\
&+c_2^3(-408c_3^2+92c_5)+c_2(135c_3^3-64c_4^2-118c_3c_5))e_n^8+O(e_n^9)\\
\end{split}
\end{equation*}
and hence
\begin{equation*}\label{c6}
\begin{split}
e_{n,y}=y_n-x^*=
&c_{2}e_{n}^{2}-2(c_{2}^{2}-c_{3})e_{n}^{3}+(4c_2^3-7c_2c_3+3c_4)e_n^4-(8c_2^4-20c_2^2c_3+6c_3^2+10c_2c_4-4c_5)e_n^5\\
&-(-16c_2^5+52c_2^3c_3-28c_2^2c_4+17c_3c_4+c_2(-33c_3^2+13c_5))e_n^6\\
&-2(16c_2^6-64c_2^4c_3-9c_3^3+36c_2^3c_4-46c_2c_3c_4+6c_4^2+9c_2^2(7c_3^2-2c_5)+11c_3c_5)e_n^7\\
&-(-64c_2^7+304c_2^5c_3-176c_2^4c_4+348c_2^2c_3c_4+c_4(-75c_3^2+31c_5)\\
&+c_2^3(-408c_3^2+92c_5)+c_2(135c_3^3-64c_4^2-118c_3c_5))e_n^8+O(e_n^9)\\
\end{split}
\end{equation*}

For $f(y_n)$, we also have
\begin{equation}\label{b7}
f(y_{n})=f^{'}(x^{*})\left(e_{n,y} +
c_{2}e_{n,y}^{2}+c_{3}e_{n,y}^{3}+\ldots+c_8e_{n,y}^{8}\right)+O(e_{n,y}^{9}).
\end{equation}

According to the proof of Theorem \ref{theorem1}, we have
\begin{equation*}
\begin{split}
e_{n,z}=z_n-x^{*}&=(-c_2c_3)e_{n}^{4}+(20c_2^4+2c_2^2c_3-2c_3^2-2c_2c_4)e_n^5+(-218c_2^5+156c_2^3c_3\\
&+3c_2^2c_4-7c_3c_4+c_2(6c_3^2-3c_5))e_n^6+2(730c_2^6-1006c_2^4c_3+2c_3^3\\
&+118c_2^3c_4+8c_2c_3c_4-3c_4^2-5c_3c_5+2c_2^2(115c_3^2-c_5))e_n^7\\
&+(-7705c_2^7+15424c_2^5c_3-2946c_2^4c_4+1393c_2^2c_3c_4+c_4(14c_3^2-17c_5)\\
&+35c_2^3(-211c_3^2+9c_5)+5c_2(121c_3^3+2c_4^2+4c_3c_5))e_n^8+O(e_n^9).
\end{split}
\end{equation*}

Also for $f(z_n)$, we get

\begin{equation}\label{b8}
f(z_{n})=f^{'}(x^{*})\left(e_{n,z} +
c_{2}e_{n,z}^{2}+c_{3}e_{n,z}^{3}+\ldots+c_8e_{n,z}^{8}\right)+O(e_{n,z}^{9}).
\end{equation}

By (\ref{b3}) and (\ref{b7}), we obtain
\begin{equation}\label{c8}
\begin{split}
t_n=\frac{f(y_n)}{f(x_n)}&=c_2
e_n+(-3c_2^2+2c_3)e_n^2+(8c_2^3-10c_2c_3+3c_4)e_n^3+(-20c_2^4+37c_2^2c_3\\
&-8c_3^2-14c_2c_4+4c_5)e_n^4+(48c_2^5-118c_2^3c_3+51c_2^2c_4-22c_3c_4+c_2(55c_3^2\\
&-18c_5))e_n^5+(-112c_2^6+344c_2^4c_3+26c_3^3-163c_2^3c_4+150c_2c_3c_4-15c_4^2\\
&-28c_3c_5+c_2^2(-252c_3^2-65c_5))e_n^6+O(e_n^7).
\end{split}
\end{equation}
\\
By (\ref{b7}) and (\ref{b8}), we obtain
\begin{equation}\label{c9}
\begin{split}
s_n=\frac{f(z_n)}{f(y_n)}&=-c_3e_n^2+(20c_2^3-2c_4)e_n^3+(-178c_2^4+121c_2^2c_3-c_3^2-c_2c_4)e_n^4+(1004c_2^5\\
&-1286c_2^3c_3+184c_2^2c_4-6c_3c_4+c_2(240c_3^2-2c_5))e_n^5+O(e_n^6).
\end{split}
\end{equation}
\\
By (\ref{b3}) and (\ref{b8}), we get
\begin{equation}\label{c10}
\begin{split}
u_n=\frac{f(z_n)}{f(x_n)}&=-c_2c_3e_n^3+(20c_2^4+3c_2^2c_3-2c_3^2-2c_2c_4)e_n^4+(-238c_2^5+153c_2^3c_3+5c_2^2c_4\\
&-7c_3c_4+c_2(9c_3^2-3c_5))e_n^5+O(e_n^6).
\end{split}
\end{equation}

Expanding $H$ at $(0,0,0)$ yields
\begin{equation}\label{b12}
\begin{split}
H(t_n,s_n,u_n)&=H_{0,0,0}+t_nH_{1,0,0}+s_nH_{0,1,0}+u_nH_{0,0,1}+t_ns_nH_{1,1,0}+t_nu_nH_{1,0,1}\\
&+s_nu_nH_{0,1,1}+t_ns_nu_nH_{1,1,1}+\frac{t_n^2}{2}H_{2,0,0}+O(t_n^3,
s_n^2, u_n^2).
\end{split}
\end{equation}

By substituting (\ref{b3})-(\ref{b12})
 into (\ref{b2}), we obtain
\begin{equation*}
e_{n+1}=x_{n+1}-x^*=R_4e_n^4+R_5e_n^5+R_6e_n^6+R_7e_n^7+R_8e_n^8+O(e_n^9),
\end{equation*}
where
\begin{equation*}
\begin{split}
R_4&=\frac{1}{2}c_2\left(2c_3+c_2^2(-10+G_2)\right)(-1+H_{0,0,0}),\\
R_5&=c_2^2c_3\left(-2+H_{1,0,0}\right),\\
R_6&=\frac{1}{2}c_2c_3\left(-2c_3\left(-1+H_{0,1,0}\right)+c_2^2\left(-12+H_{2,0,0}\right)\right),\\
R_7&=\frac{-1}{6}c_2^2c_3\left(c_2^2\left(36+G_3\right)+6c_3\left(-4+H_{0,0,1}+H_{1,1,0}\right)\right),\\
R_8&=c_2c_3\left(12c_2^4+14c_2^2c_3-c_3^2-c_2c_4\right).
\end{split}
\end{equation*}

Clearly $R_{8}\neq 0$, by setting $R_4 = R_5 =R_6 = R_7 = 0$, the
convergence order becomes eight. Sufficient conditions are given
by the following set of equations
\begin{equation*}
  \begin{array}{lrl}
    H_{0,0,0}=1, \quad G_2=10, & \quad \Rightarrow & R_4 = 0 ,
  \\[1ex]
    H_{1,0,0}=2, & \Rightarrow & R_5=0 ,
  \\[1ex]
   H_{0,1,0}=1, \quad H_{2,0,0}=12, & \quad \Rightarrow & R_6 = 0 ,
  \\[1ex]
    G_3=-36, \quad H_{0,0,1}=4, \quad H_{1,1,0}=0, \quad  & \Rightarrow & R_7=0 ,
  \end{array}
\end{equation*}
and the error equation becomes
\begin{equation*}
  e_{n+1}=R_{8}e_n^8+O(e_n^9),
\end{equation*}
which finishes the proof.
\end{proof}
\subsection{Main contribution:\\
Construction of an optimal four-point class}

This section contains the main contribution. To
increase the order of convergence, we add one Newton's step to the
method \eqref{b2} to get
\begin{equation}\label{c1}
\begin{cases}
y_n=x_n-\dfrac{f(x_n)}{f'(x_n)},\\
z_{n}=y_n-G(t_n)\cdot\dfrac{f(y_n)}{f'(x_n)},\\
w_{n}=z_n-H(t_n,s_n,u_n)\cdot\dfrac{f(z_n)}{f'(x_n)},\\
x_{n+1}=w_n-\dfrac{f(w_n)}{f'(w_n)}.
\end{cases}
\end{equation}
The method \eqref{c1} uses six function evaluations therefore the
method is not an optimal four-point method. We modify \eqref{c1}
by approximating $f'(w_n)$ by
\begin{equation*}
  f'(w_n) \approx \frac{f'(x_n)}{I(t_n)+J(s_n)+K(u_n)+L(t_n,u_n)+M(p_n,q_n,r_n)+N(t_n,s_n,u_n,r_n)},
\end{equation*}
using only the values $f(x_n)$, $f(y_n)$, $f(z_n)$, $f(w_n)$ and
$f'(x_n)$. More precisely, we use the abbreviations
$t_n=\frac{f(y_n)}{f(x_n)}$, $s_n=\frac{f(z_n)}{f(y_n)}$,
$u_n=\frac{f(z_n)}{f(x_n)}$, $p_n=\frac{f(w_n)}{f(x_n)}$,
$q_n=\frac{f(w_n)}{f(y_n)}$, $r_n=\frac{f(w_n)}{f(z_n)}$ and
utilize Mathematica \cite{Hazrat} to carefully choose the weight
functions $I,J,K:\mathbb{R}\rightarrow \mathbb{R}$,
$L:\mathbb{R}^2\rightarrow\mathbb{R}$,
$M:\mathbb{R}^3\rightarrow\mathbb{R}$ and $N:\mathbb{R}^4
\rightarrow \mathbb{R}$ from a class of admissible functions such
that the scheme
\begin{equation}\label{c2}
\begin{cases}
y_n=x_n-\dfrac{f(x_n)}{f'(x_n)},\\
z_{n}=y_n- G(t_n) \cdot\dfrac{f(y_n)}{f'(x_n)},\\
w_{n}=z_n-H(t_n, s_n, u_n)\cdot\dfrac{f(z_n)}{f'(x_n)},\\
x_{n+1}=w_n-\left(I(t_n)+J(s_n)+K(u_n)+L(t_n,u_n)+M(p_n,q_n,r_n)+N(t_n,s_n,u_n,r_n)\right)\cdot\dfrac{f(w_n)}{f'(x_n)},
\end{cases}
\end{equation}
is of order $16$.

\begin{thm}\label{theorem3}
Assume that $f: D \subset \mathbb{R} \rightarrow
\mathbb{R}$ is $16$ times continuously differentiable and has a simple zero
$x^{*}\in D$ and let $I,J,K:\mathbb{R}\rightarrow \mathbb{R}$,
$L:\mathbb{R}^2\rightarrow\mathbb{R}$,
$M:\mathbb{R}^3\rightarrow\mathbb{R}$ and $N:\mathbb{R}^4
\rightarrow \mathbb{R}$ be
sufficiently  differentiable functions. If
\begin{equation*}
  \begin{split}
   &I_{0}=0,\quad I_{1}=2,\quad I_{2}=12,\\
   &J_{0}=0,\quad J_{1}=1,\quad J_{2}=0,\quad J_{3}=-6\\
   &K_{0}=1,\quad K_{1}=4,\quad K_{2}=-8,\\
   &L_{0,0}=0,\quad L_{1,0}=0,\quad L_{1,1}=1,\quad L_{2,0}=0,\quad L_{0,1}=0,\\
   &L_{3,0}=0, \quad L_{2,1}=12, \quad L_{3,1}=12, \quad L_{0,2}=0, \quad L_{1,2}=-20,\\
   &H_{0,1,1}=0, \quad H_{1,1,1}=0,\\
   &M_{0,0,0}=0,\quad M_{1,0,0}=8, \quad M_{0,1,0}=2, \quad M_{0,0,1}=1,\\
   &N_{0,0,0,0}=N_{1,0,0,0}=N_{0,1,0,0}=N_{0,0,1,0}=N_{0,0,0,1}=N_{2,0,0,0}=N_{1,1,0,0}\\
   &\quad \quad \quad =N_{0,2,0,0}=N_{1,0,1,0}=N_{2,1,0,0}=N_{3,1,0,0}=N_{2,2,0,0}=N_{0,1,0,1}\\
   &\quad \quad \quad =N_{3,0,0,0}=N_{4,0,0,0}=N_{3,0,1,0}=N_{1,1,0,1}=N_{2,1,1,0}=N_{3,0,0,1}\\
   &\quad \quad \quad =N_{0,0,1,1}=N_{3,2,0,0}=N_{4,1,0,0}=N_{1,2,0,0}=N_{2,0,1,0}=N_{1,1,1,0}=0,\\
   & N_{1,0,0,1}=2,\quad N_{2,0,0,1}=12, \quad N_{0,2,1,0}=-8,\quad N_{4,0,1,0}=576, \quad N_{0,1,1,0}=2.\\
  \end{split}
\end{equation*}
where $I_{i}=\frac{d^{i}I(t_n)}{d t_n^i}|_{0}$,
$J_{i}=\frac{d^{i}J(s_n)}{d s_n^i}|_{0}$ and
$K_{i}=\frac{d^{i}K(u_n)}{d u_n^i}|_{0}$ ,
$L_{i,j}=\frac{\partial^{i+j}L(t_n,u_n)}{\partial t_n^i\partial
u_n^j}|_{(0,0)}$,
$M_{i,j,k}=\frac{\partial^{i+j+k}M(p_n,q_n,r_n)}{\partial
p_n^i\partial q_n^j\partial r_n^k}|_{(0,0,0)}$  and
$N_{i,j,k,l}=\frac{\partial^{i+j+k+l}N(t_n,s_n,u_n,r_n)}{\partial
t_n^i\partial s_n^j\partial u_n^k\partial r_n^l}|_{(0,0,0,0)}$ for
$i,j,k,l=0,1,2,3,\ldots$, and the initial point $x_{0}$ is
sufficiently close to $x^{*}$, then the sequence $\{x_n\}$ defined
by \eqref{c2} converges to $x^{*}$ and the order of convergence is
$16$.
\end{thm}

\begin{proof}
Let $e_{n}:=x_{n}-x^{*}$, $e_{n,y}:=y_{n}-x^{*}$,
$e_{n,z}:=z_{n}-x^{*}$, $e_{n,w}:=w_{n}-x^{*}$
 and
$c_{n}:=\frac{f^{(n)}(x^{*})}{n!f^{'}(x^{*})}$ for $n \in
\mathbb{N}$. Using the fact that $f(x^{*})=0$,  Taylor's
expansion of $f$ at $x^*$ yields
\begin{equation}\label{c3}
f(x_{n})=f^{'}(x^*)(e_{n} +
c_{2}e_{n}^{2}+c_{3}e_{n}^{3}+\ldots+c_{16}e_{n}^{16})+O(e_{n}^{17}),
\end{equation}
and expanding $f'$ at $x^*$ we get
\begin{equation}\label{c4}
f^{'}(x_{n})=f^{'}(x^{*})(1+2c_{2}e_{n}+3c_{3}e_{n}^{2}+\ldots+17c_{17}e_{n}^{16})+O(e_{n}^{17}).
\end{equation}
Therefore
\begin{equation*}\label{c5}
\begin{split}
\dfrac{f(x_{n})}{f'(x_{n})}&=e_{n}-c_{2}e_{n}^{2}+2(c_{2}^{2}-c_{3})e_{n}^{3}+(-4c_2^3+7c_2c_3-3c_4)e_n^4+(8c_2^4-20c_2^2c_3+6c_3^2\\
&+10c_2c_4-4c_5)e_n^5+2(-16c_2^5+52c_2^3c_3-28c_2^2c_4+17c_3c_4+c_2(-33c_3^2+13c_5))e_n^6\\
&+\ldots+(-64c_2^7+304c_2^5c_3-176c_2^4c_4+348c_2^2c_3c_4+c_2^3(-408c_3^2+92c_5)\\
&+c_4(-75c_3^2+31c_5)+c_2(135c_3^3-64c_4^2-118c_3c_5))e_n^8+O(e_n^9),
\end{split}
\end{equation*}
so,
\begin{equation*}\label{c6}
\begin{split}
e_{n,y}=
y_n-x^*&=c_{2}e_{n}^{2}+(-2c_2^2+2c_3)e_n^3+(4c_2^3-7c_2c_3+3c_4)e_n^4+(-8c_2^4+20c_2^2c_3-6c_3^2\\
&-10c_2c_4+4c_5)e_n^5+2(16c_2^5-52c_2^3c_3+28c_2^2c_4-17c_3c_4+c_2(33c_3^2-13c_5))e_n^6\\
&+\ldots+(64c_2^7-304c_2^5c_3+176c_2^4c_4-348c_2^2c_3c_4+c_2^3(408c_3^2-92c_5)\\
&+c_4(75c_3^2-31c_5)+c_2(-135c_3^3+64c_4^2+118c_3c_5))e_n^8+O(e_n^9).
\end{split}
\end{equation*}

For $f(y_n)$, we also have
\begin{equation}\label{c7}
f(y_{n})=f^{'}(x^{*})\left(e_{n,y} +
c_{2}e_{n,y}^{2}+c_{3}e_{n,y}^{3}+\ldots+c_{16}e_{n,y}^{16}\right)+O(e_{n,y}^{17}).
\end{equation}

According to the proof of Theorem \ref{theorem2}, we have
\begin{equation*}
\begin{split}
e_{n,z}=z_n-x^{*}&=(-c_2c_3)e_{n}^{4}+(20c_2^4+2c_2^2c_3-2c_3^2-2c_2c_4)e_n^5+(-218c_2^5+156c_2^3c_3\\
&+3c_2^2c_4-7c_3c_4+c_2(6c_3^2-3c_5))e_n^6+2(730c_2^6-1006c_2^4c_3+2c_3^3\\
&+118c_2^3c_4+8c_2c_3c_4-3c_4^2-5c_3c_5+2c_2^2(115c_3^2-c_5))e_n^7\\
&+(-7705c_2^7+15424c_2^5c_3-2946c_2^4c_4+1393c_2^2c_3c_4+c_4(14c_3^2-17c_5)\\
&+35c_2^3(-211c_3^2+9c_5)+5c_2(121c_3^3+2c_4^2+4c_3c_5))e_n^8+O(e_n^9).
\end{split}
\end{equation*}

For $f(z_n)$, we also get
\begin{equation}\label{cc7}
f(z_{n})=f^{'}(x^{*})\left(e_{n,z} +
c_{2}e_{n,z}^{2}+c_{3}e_{n,z}^{3}+\ldots+c_{16}e_{n,z}^{16}\right)+O(e_{n,z}^{17}).
\end{equation}

By (\ref{c3}) and (\ref{c7}), we obtain
\begin{equation}\label{c8}
\begin{split}
t_n=\frac{f(y_n)}{f(x_n)}&=c_2
e_n+(-3c_2^2+2c_3)e_n^2+(8c_2^3-10c_2c_3+3c_4)e_n^3+(-20c_2^4+37c_2^2c_3\\
&-8c_3^2-14c_2c_4+4c_5)e_n^4+(48c_2^5-118c_2^3c_3+51c_2^2c_4-22c_3c_4+c_2(55c_3^2\\
&-18c_5))e_n^5+(-112c_2^6+344c_2^4c_3+26c_3^3-163c_2^3c_4+150c_2c_3c_4-15c_4^2\\
&-28c_3c_5+c_2^2(-252c_3^2+65c_5))e_n^6+(256c_2^7-944c_2^5c_3+480c_2^4c_4-693c_2^2c_3c_4\\
&+c_2^3(952c_3^2-207c_5)+c_4(105c_3^2-38c_5)+2c_2(-114c_3^3+51c_4^2+95c_3c_5))e_n^7\\
&+(-576c_2^8+2480c_2^6c_3-1336c_2^5c_4+2660c_2^3c_3c_4+6c_2c_4(-156c_3^2+43c_5)\\
&+c_2^4(-3200c_3^2+607c_5)+3c_2^2(418c_3^3-159c_4^2-292c_3c_5)+3(-24c_3^4+47c_3c_4^2\\
&+44c_3^2c_5-8c_5))e_n^8+O(e_n^9).
\end{split}
\end{equation}

 By (\ref{c7}) and
(\ref{cc7}), we obtain
\begin{equation}\label{c9}
\begin{split}
s_n=\frac{f(z_n)}{f(y_n)}&=-c_3e_n^2+(20c_2^3-2c_4)e_n^3+(-178c_2^4+121c_2^2c_3-c_3^2-c_2c_4-3c_5)e_n^4+(1004c_2^5\\
&-1286c_2^3c_3+184c_2^2c_4-6c_3c_4+c_2(240c_3^2-2c_5))e_n^5+(-4567c_2^6+8541c_2^4c_3\\
&+155c_3^3+1863c_2^3c_4+725c_2c_3c_4-7c_4^2-10c_3c_5+\frac{1}{2}c_2^2(-6866c_3^2+492c_5))e_n^6\\
&+\frac{1}{6}(109500c_2^7-269472c_2^5c_3+72756c_2^4c_4-59376c_2^2c_3c_4+c_2^3(172248c_3^2-14616c_5)\\
&+12c_4(348c_3^2-11c_5)+c_2(-24048c_3^3+3276c_4^2+5796c_3c_5))e_n^7+(-66359c_2^8\\
&+204034c_2^6c_3-1725c_3^4-63159c_2^5c_4+81203c_2^3c_3c_4+1038c_3c_4^2+923c_3^2c_5-17c_5^2\\
&+c_2c_4(-17234c_3^2+1453c_5)+c_2^4(-181979c_3^2+15700c_5)+2c_2^2(23804c_3^3-3559c_4^2\\
&-6452c_3c_5))e_n^8+O(e_n^9).
\end{split}
\end{equation}

By (\ref{c3}) and (\ref{cc7}), we have
\begin{equation}\label{c10}
\begin{split}
u_n=\frac{f(z_n)}{f(x_n)}&=-c_2c_3e_n^3+(20c_2^4+3c_2^2c_3-2c_3^2-2c_2c_4)e_n^4+(-238c_2^5+153c_2^3c_3+5c_2^2c_4\\
&-7c_3c_4+c_2(9c_3^2-3c_5))e_n^5+(1698c_2^6-2185c_2^4c_3+6c_3^3+231c_2^3c_4+26c_2c_3c_4\\
&-6c_4^2-10c_3c_5+7c_2^2(64c_3^2+c_5))e_n^6+(-9403c_2^7+17847c_2^5c_3-3197c_2^4c_4\\
&+1359c_2^2c_3c_4+c_4(23c_3^2-17c_5)+c_2^3(-7985c_3^2+308c_5)+2c_2(295c_3^3+9c_4^2\\
&+17c_3c_5))e_n^7+(44503c_2^8-111251c_2^6c_3+292c_3^4+25635c_2^5c_4-23315c_2^3c_3c_4\\
&+27c_3c_4^2+382c_2^4(203c_3^2-11c_5)+28c_3^2c_5-12c_5^2+2c_2c_4(1344c_3^2+23c_5)\\
&+c_2^2(-14492c_3^3+1031c_4^2+1816c_3c_5))e_n^8+O(e_n^9),
\end{split}
\end{equation}

According to the proof of Theorem \ref{theorem2}, we have
\begin{equation}\label{cc6}
\begin{split}
e_{n,w}=
w_n-x^*&=c_2c_3(12c_2^4+14c_2^2c_3-c_3^2-c_2c_4)e_n^8+2(120c_2^8+344c_2^6c_3-37c_2^4c_3^2+c_3^4\\
&-22c_2^5c_4-30c_2^3c_3c_4+5c_2c_3^2c_4+c_2^2(-40c_3^3+c_4^2+c_3c_5))e_n^9+(10296c_2^9+5059c_2^7c_3\\
&-1622c_2^6c_4+452c_2^4c_3c_4-19c_3^3c_4+c_2^2c_4(456c_3^2-7c_5)+c_2^3(236c_3^3+63c_4^2\\
&+91c_3c_5))e_n^{10}+O(e_n^{11}).
\end{split}
\end{equation}

For $f(w_n)$, we also obtain
\begin{equation}\label{ccc7}
f(w_{n})=f^{'}(x^{*})\left(e_{n,w} +
c_{2}e_{n,w}^{2}+c_{3}e_{n,w}^{3}+\ldots+c_{16}e_{n,w}^{16}\right)+O(e_{n,w}^{17}).
\end{equation}

By (\ref{c3}) and (\ref{ccc7}), we have
\begin{equation}\label{c11}
\begin{split}
p_n=\frac{f(w_n)}{f(x_n)}&=c_2c_3(12c_2^4+14c_2^2c_3-c_3^2-c_2c_4)e_n^7+(-240c_2^8-700c_2^6c_3+60c_2^4c_3^2-2c_3^4\\
&+44c_2^5c_4+61c_2^3c_3c_4-10c_2c_3^2c_4+c_2^2(81c_3^3-2c_4^2-2c_3c_5))e_n^8+(10536c_2^9\\
&+5759c_2^7c_3-1666c_2^6c_4+391c_2^4c_3c_4-19c_3^3c_4+c_2^2c_4(467c_3^2-7c_5)+c_2^5(-8121c_3^2\\
&+76c_5)+c_2c_3(151c_3^3-26c_4^2-17c_3c_5)+c_2^3(141c_3^3+65c_4^2+93c_3c_5))e_n^9+O(e_n^{10}).
\end{split}
\end{equation}

By (\ref{c7}) and (\ref{ccc7}), we get
\begin{equation}\label{c12}
\begin{split}
q_n=\frac{f(w_n)}{f(y_n)}&=-c_3(-12c_2^4-14c_2^2c_3+c_3^2+c_2c_4)e_n^6-2(120c_2^7+332c_2^5c_3-39c_2^3c_3^2-22c_2^4c_4\\
&-29c_2^2c_3c_4+4c_3^2c_4+c_2(-25c_3^3+c_4^2+c_3c_5))e_n^7+(9816c_2^8+4151c_2^6c_3-1534c_2^5c_4\\
&+449c_2^3c_3c_4+c_2c_4(275c_3^2-7c_5)+c_2^4(-6551c_3^2+76c_5)+c_3(41c_3^3-19c_4^2-13c_3c_5)\\
&+c_2^2(288c_3^3+59c_4^2+87c_3c_5))e_n^8+O(e_n^9).
\end{split}
\end{equation}

By (\ref{cc7}) and (\ref{ccc7}), we obtain
\begin{equation}\label{c13}
\begin{split}
r_n=\frac{f(w_n)}{f(z_n)}&=(-12c_2^4-14c_2^2c_3+c_3^2+c_2c_4)e_n^4+2(384c_2^5-58c_2^3c_3c_3-30c_2^2c_4+6c_3c_4\\
&+2c_2(-25c_3^2+c_5))e_n^5+(-4271c_2^6+3691c_2^4c_3-42c_3^3-78c_2^3c_4-177c_2c_3c_4+7c_4^2\\
&+10c_3c_5-c_2^2(148c_3^2+45c_5))e_n^6+2(15680c_2^7-23481c_2^5c_3+2760c_2^4c_4-237c_2^2c_3c_4\\
&+c_2^3(7182c_3^2-51c_5)+c_4(-102c_3^2+11c_5)-c_2(126c_3^3+75c_4^2+125c_3c_5))e_n^7\\
&+(-182548c_2^8+384369c_2^6c_3-191c_3^4-68503c_2^5c_4+43033c_2^3c_3c_4-321c_3c_4^2\\
&-278c_3^2c_5+17c_5^2-c_2c_4(1316c_3^2+417c_5)+c_2^4(-213934c_3^2+7347c_5)\\
&+c_2^2(28040c_3^3-398c_4^2-673c_3c_5)e_n^8+O(e_n^9).
\end{split}
\end{equation}
Expanding $I,J, K, L,M$ and $N$ at $0$ in $\mathbb{R}, \mathbb{R}^2, \mathbb{R}^3$ and $\mathbb{R}^4$, respectively,
yield

\begin{equation}\label{c14}
I(t_n)=I_{0}+t_n I_{1}+ \frac{t_n^2}{2}I_{2}+O(t_n^3),
\end{equation}

\begin{equation}\label{c15}
J(s_n)=J_{0}+s_n J_{1}+
\frac{s_n^2}{2}J_{2}+\frac{s_n^3}{6}J_{3}+O(s_n^4),
\end{equation}

\begin{equation}\label{c16}
K(u_n)=K_{0}+u_nK_{1}+\frac{u_n^2}{2}K_{2}+ O(u_n^3),
\end{equation}

\begin{equation}\label{c17}
\begin{split}
L(t_n,u_n)&=L_{0,0}+t_nL_{1,0}+u_nL_{0,1}+t_nu_nL_{1,1}+\frac{t_n^2}{2}L_{2,0}+\frac{u_n^2}{2}L_{0,2}+\frac{t_n
u_n^2}{2}L_{1,2}\\
&+\frac{t_n^2u_n}{2}L_{2,1}+\frac{t_n^2u_n^2}{4}L_{2,2}+\frac{t_n^3}{6}L_{3,0}+\frac{t_n^3u_n}{6}L_{3,1}+\frac{t_n^3u_n^2}{12}L_{3,2}+O(t_n^4,u_n^3),
\end{split}
\end{equation}

\begin{equation}\label{c18}
\begin{split}
M(p_n,q_n,r_n)&=M_{0,0,0}+p_nM_{1,0,0}+q_nM_{0,1,0}+r_nM_{0,0,1}+p_nq_nM_{1,1,0}+p_nr_nM_{1,0,1}+p_nq_nr_nM_{1,1,1}\\
&+q_nr_nM_{0,1,1}+\frac{p_n^2}{2}M_{2,0,0}+\frac{q_n^2}{2}M_{0,2,0}+\frac{r_n^2}{2}M_{0,0,2}+\frac{p_nr_n^2}{2}M_{1,0,2}+\frac{p_n^2r_n}{2}M_{2,0,1}\\
&+\frac{p_n^2q_n}{2}M_{2,1,0}+\frac{p_n
q_n^2}{2}M_{1,2,0}+\frac{p_n
r_n^2}{2}M_{1,0,2}+\frac{q_nr_n^2}{2}M_{0,1,2}+\frac{q_n^2r_n}{2}M_{0,2,1}\\
&+\frac{p_nq_nr_n^2}{2}M_{1,1,2}+\frac{p_nq_n^2r_n}{2}M_{1,2,1}+\frac{p_n^2q_nr_n}{2}M_{2,1,1}+\frac{q_n^2r_n^2}{4}M_{0,2,2}+\frac{p_nq_n^2r_n^2}{4}M_{1,2,2}\\
&+\frac{p_n^2r_n^2}{4}M_{2,0,2}+\frac{p_n^2q_nr_n^2}{4}M_{2,1,2}+\frac{p_n^2q_n^2}{4}M_{2,2,0}+\frac{p_n^2q_n^2r_n}{4}M_{2,2,1}+\frac{p_n^2q_n^2r_n^2}{8}M_{2,2,2}\\
&+O(p_n^3,q_n^3,r_n^3),\\
\end{split}
\end{equation}
and,
\begin{equation}\label{c19}
\begin{split}
N(t_n,s_n,u_n,r_n)&=N_{0,0,0,0}+t_nN_{1,0,0,0}+s_nN_{0,1,0,0}+u_nN_{0,0,1,0}+r_nN_{0,0,0,1}+t_ns_nN_{1,1,0,0}\\
&+t_nu_nN_{1,0,1,0}+t_nr_nN_{1,0,0,1}+u_nr_nN_{0,0,1,1}+s_nu_nN_{0,1,1,0}+s_nr_nN_{0,1,0,1}\\
&+t_ns_nu_nN_{1,1,1,0}+t_ns_nr_nN_{1,1,0,1}+s_nu_nr_nN_{0,1,1,1}+t_nu_nr_nN_{1,0,1,1}\\
&+t_ns_nu_nr_nN_{1,1,1,1}+\frac{t_n^2}{2}N_{2,0,0,0}+\frac{s_n^2}{2}N_{0,2,0,0}+\frac{s_n^2r_n}{2}N_{0,2,0,1}+\frac{s_n^2u_n}{2}N_{0,2,1,0}\\
&+\frac{s_n^2u_nr_n}{2}N_{0,2,1,1}+\frac{t_ns_n^2}{2}N_{1,2,0,0}+\frac{t_ns_n^2r_n}{2}N_{1,2,0,1}+\frac{t_ns_n^2u_n}{2}N_{1,2,1,0}\\
&+\frac{t_n^2r_n}{2}N_{2,0,0,1}+\frac{t_n^2u_n}{2}N_{2,0,1,0}+\frac{t_n^2s_n}{2}N_{2,1,0,0}+\frac{t_n^2u_nr_n}{2}N_{2,0,1,1}\\
&+\frac{t_n^2s_n}{2}N_{2,1,0,0}+\frac{t_n^2s_nr_n}{2}N_{2,1,0,1}+\frac{t_n^2s_nu_n}{2}N_{2,1,1,0}+\frac{t_n^2s_nu_nr_n}{2}N_{2,1,1,1}\\
&+\frac{t_n^2s_n^2}{4}N_{2,2,0,0}+\frac{t_n^2s_n^2r_n}{4}N_{2,2,0,1}+\frac{t_n^2s_n^2u_n}{4}N_{2,2,1,0}+\frac{t_n^2s_n^2u_nr_n}{4}N_{2,2,1,1}\\
&+\frac{t_n^3}{6}N_{3,0,0,0}+\frac{t_n^3r_n}{6}N_{3,0,0,1}+\frac{t_n^3u_n}{6}N_{3,0,1,0}+\frac{t_n^3u_nr_n}{6}N_{3,0,1,1}\\
&+\frac{t_n^3s_n}{6}N_{3,1,0,0}+\frac{t_n^3s_nr_n}{6}N_{3,1,0,1}+\frac{t_n^3s_nu_n}{6}N_{3,1,1,0}+\frac{t_n^3s_nu_nr_n}{6}N_{3,1,1,1}\\
&+\frac{t_n^3s_n^2}{12}N_{3,2,0,0}+\frac{t_n^3s_n^2r_n}{12}N_{3,2,0,1}+\frac{t_n^3s_n^2u_n}{12}N_{3,2,1,0}+\frac{t_n^3s_n^2u_nr_n}{12}N_{3,2,1,1}\\
&+\frac{t_n^4}{24}N_{4,0,0,0}+\frac{t_n^4r_n}{24}N_{4,0,0,1}+\frac{t_n^4u_n}{24}N_{4,0,1,0}+\frac{t_n^4s_nr_n}{24}N_{4,1,0,1}+\frac{t_n^4s_n}{24}N_{4,1,0,0}\\
&+\frac{t_n^4s_nr_n}{24}N_{4,1,0,1}+\frac{t_n^4s_nu_n}{24}N_{4,1,1,0}+\frac{t_n^4s_nu_nr_n}{24}N_{4,1,1,1}+\frac{t_n^4s_n^2}{48}N_{4,2,0,0}\\
&+\frac{t_n^4s_n^2r_n}{48}N_{4,2,0,1}+\frac{t_n^4s_n^2u_n}{48}N_{4,2,1,0}+\frac{t_n^4s_n^2u_nr_n}{48}N_{4,2,1,1}+O(t_n^5,s_n^3,u_n^2,r_n^2).\\
\end{split}
\end{equation}
By substituting (\ref{c3})-(\ref{c19})
 into (\ref{c2}), we obtain
\begin{equation*}
\begin{split}
e_{n+1}=x_{n+1}-x^*&=R_8e_n^8+R_9e_n^9+R_{10}e_n^{10}+R_{11}e_n^{11}+R_{12}e_n^{12}\\
&+R_{13}e_n^{13}+R_{14}e_n^{14}+R_{15}e_n^{15}+R_{16}e_n^{16}+O(e_n^{17}),
\end{split}
\end{equation*}
where
\begin{equation*}
\begin{split}
R_8&=-c_2c_3\left(12c_2^4+14c_2^2c_3-c_3^2-c_2c_4\right)\left(-1+I_{0}+J_{0}+K_{0}+L_{0,0}+M_{0,0,0}+N_{0,0,0,0}\right)\\
\\
R_9&=c_2^2c_3\left(-12c_2^4-14c_2^2c_3+c_3^2+c_2c_4\right)\left(-2+I_{1}+L_{1,0}+N_{1,0,0,0}\right),\\
\\
R_{10}&=\frac{-1}{2}c_2c_3\left(12c_2^4+14c_2^2c_3-c_3^2-c_2c_4\right)\\
&\big(-2c_3\left(-1+J_{1}+N_{0,1,0,0}\right)+c_2^2(-12+I_{2}+L_{2,0}+N_{2,0,0,0})\big),\\
\end{split}
\end{equation*}

\begin{equation*}
\begin{split}
R_{11}&=\frac{-1}{6}c_2^2c_3\left(12c_2^4+14c_2^2c_3-c_3^2-c_2c_4\right)\big(-6c_3\left(-4+K_1+L_{0,1}+N_{0,0,1,0}+N_{1,1,0,0}\right)\\
&+c_2^2\left(L_{3,0}+N_{3,0,0,0}\right)\big),\\
\\
R_{12}&=\frac{1}{24}c_2c_3\left(12c_2^4+14c_2^2c_3-c_3^2-c_2c_4\right)\\
&(-24c_2c_4(-1+M_{0,0,1}+N_{0,0,0,1})-12c_3^2(J_2+2(-1+M_{0,0,1}+N_{0,0,0,1})+N_{0,2,0,0})\\
&+12c_2^2c_3(2(-15+L_{1,1}+14M_{0,0,1}+14N_{0,0,0,1}+N_{1,0,1,0})+N_{2,1,0,0})\\
&+c_2^4(288(-1+M_{0,0,1}+N_{0,0,0,1})-N_{4,0,0,0})),\\
\\
R_{13}&=\frac{1}{6}c_2^2c_3(12c_2^4+14c_2^2c_3-c_3^2-c_2c_4)\\
&(72c_2^4(-2+N_{1,0,0,1})-6c_2c_4(-2+N_{1,0,0,1})+3c_3^2(2H_{0,1,1}-2(-4+N_{0,1,1,0}+N_{1,0,0,1})-N_{1,2,0,0})\\
&+c_2^2c_3(3(-68+L_{2,1}+28N_{1,0,0,1}+N_{2,0,1,0})+N_{3,1,0,0})),\\
\end{split}
\end{equation*}

\begin{equation*}
\begin{split}
R_{14}&=\frac{1}{24}c_2c_3(12c_2^4+14c_2^2c_3-c_3^2-c_2c_4)\\
&(24c_2c_3c_4(-2+M_{0,1,0}+N_{0,1,0,1})+4c_3^3(J_3+6(-1+M_{0,1,0}+N_{0,1,0,1}))\\
&+144c_2^6(-12+N_{2,0,0,1})-12c_2^3c_4(-12+N_{2,0,0,1})\\
&-6c_2^2c_3^2(2(-60+K_2+L_{0,2}-2H_{1,1,1}+28M_{0,1,0}+28N_{0,1,0,1}+2N_{1,1,1,0}+N_{2,0,0,1})+N_{2,2,0,0})\\
&+c_2^4c_3(4(-372+L_{3,1}-72M_{0,1,0}-72N_{0,1,0,1}+42N_{2,0,0,1}+N_{3,0,1,0})+N_{4,1,0,0})),\\
\end{split}
\end{equation*}

\begin{equation*}
\begin{split}
R_{15}&=\frac{1}{24}c_2^2c_3(12c_2^4+14c_2^2c_3-c_3^2-c_2c_4)\\
&(24c_2c_3c_4(-8+M_{1,0,0}+N_{0,0,1,1}+N_{1,1,0,1})+12c_3^3(-8+2M_{1,0,0}+2N_{0,0,1,1}+N_{0,2,1,0}+2N_{1,1,0,1})\\
&+48c_2^6N_{3,0,0,1}-4c_2^3c_4N_{3,0,0,1}-2c_2^2c_3^2(6(L_{1,2}+4(-51+7M_{1,0,0}+7N_{0,0,1,1}+7N_{1,1,0,1})+N_{2,1,1,0}\\
&+2N_{3,0,0,1}+N_{3,2,0,0})+c_2^4c_3(-288(-6+M_{1,0,0}+N_{0,0,1,1}+N_{1,1,0,1})+56N_{3,0,0,1}+N_{4,0,1,0})),\\
\end{split}
\end{equation*}

\begin{equation*}
\begin{split}
R_{16}&=\frac{-1}{48}c_2c_3(12c_2^4+14c_2^2c_3-c_3^2-c_2c_4)\\
&(24c_3^4(-2+M_{0,0,2}+N_{0,2,0,1})+24c_2c_3^2c_4(-6+2M_{0,0,2}+N_{0,2,0,1})\\
&-24c_2^3c_3c_4(-88+28M_{0,0,2}+2N_{1,0,1,1}+N_{2,1,0,1})-24c_2^2(2c_3c_5-c_4^2(-2+M_{0,0,2})\\
&+c_3^3(-86+28M_{0,0,2}+14N_{0,2,0,1}+2N_{1,0,1,1}+N_{1,2,1,0}+6N_{1,2,1,0}+N_{2,1,0,1}))\\
&+4c_2^6c_3(36(-145+56M_{0,0,2}+4N_{1,0,1,1}+2N_{2,1,0,1})-7N_{4,0,0,1})\\
&+24c_2^8(144(-2+M_{0,0,2})-N_{4,0,0,1})+2c_2^5c_4(576-288M_{0,0,2}+N_{4,0,0,1})+c_2^4c_3^2\\
&(2(6L_{2,2}+4(6(86M_{0,0,2}-6N_{0,2,0,1}+7(-45+2N_{1,0,1,1}+N_{2,1,0,1}))+N_{3,1,1,0})+N_{4,0,0,1})+N_{4,2,0,0})).
\end{split}
\end{equation*}

In general $R_{16}\neq 0$, however, by setting $R_8 = R_9 =R_{10}
=R_{11}=R_{12}=R_{13}=R_{15}=R_{15} = 0$, the convergence order
becomes $16$. Sufficient conditions are given by the following
set of equations
\begin{equation*}
  \begin{array}{lrl}
    K_0=1, \quad I_0=J_0=L_{0,0}=M_{0,0,0}=N_{0,0,0,0}=0, & \Rightarrow & R_8 = 0,
  \\[4ex]
    I_1=2, \quad L_{1,0}=N_{1,0,0,0}=0, & \Rightarrow & R_9=0,
  \\[4ex]
    I_2=12, \quad J_1=1, \quad L_{2,0}=N_{0,1,0,0}=N_{2,0,0,0}=0, & \quad \Rightarrow & R_{10} = 0,
  \\[4ex]
   K_1=4, \quad L_{0,1}=L_{3,0}=N_{0,0,1,0}=N_{1,1,0,0}=N_{3,0,0,0}=0, & \Rightarrow & R_{11}=0,
    \\[4ex]
   L_{1,1}=1, \quad M_{0,0,1}=1, \quad J_2=N_{0,0,0,1}=N_{0,2,0,0}=N_{4,0,0,0}=N_{1,0,1,0}=N_{2,1,0,0}=0, & \Rightarrow & R_{12}=0 ,
    \\[4ex]
   L_{2,1}=12, \quad N_{1,0,0,1}=N_{0,1,1,0}=2, \quad H_{0,1,1}=N_{1,2,0,0}=N_{2,0,1,0}=N_{3,1,0,0}=0, & \Rightarrow & R_{13}=0 ,
    \\[4ex]
    J_3=-6, \quad K_2=-8,\quad L_{3,1}=12, \quad M_{0,1,0}=2, \quad N_{2,0,0,1}=12,
    \\[1ex]
    L_{0,2}=H_{1,1,1}=N_{0,1,0,1}=N_{1,1,1,0}=N_{2,2,0,0}=N_{3,0,1,0}=N_{4,1,0,0}=0 & \Rightarrow & R_{14}=0 ,
    \\[4ex]
    L_{1,2}=-20, \quad M_{1,0,0}=8,\quad N_{0,2,1,0}=-8 \quad N_{4,0,1,0}=576,
    \\[1ex]
   N_{0,0,1,1}=N_{1,1,0,1}=N_{3,0,0,1}=N_{2,1,1,0}=N_{3,2,0,0}=0 & \Rightarrow & R_{15}=0 ,
   \\[4ex]
    L_{2,2}=12, \quad M_{0,0,2}=2,\quad N_{1,0,1,1}=30 \quad N_{1,2,1,0}=-18,
    \\[1ex]
   N_{0,2,0,1}=N_{2,1,0,1}=N_{4,0,0,1}=N_{3,1,1,0}=N_{4,2,0,0}=0 & \Rightarrow & R_{16}\neq0 ,
  \end{array}
\end{equation*}
and the error equation becomes
\begin{equation*}
  e_{n+1}=\left(-c_2^2c_3^2(12c_2^4-c_3^2-c_2c_4)(93c_2^5-c_3c_4-c_2c_5)\right)e_n^{16}+O(e_n^{17}),
\end{equation*}
which finishes the proof.
\end{proof}
\section{Numerical results}
\subsection{Numerical implementation and comparison}

In this section, three concrete methods of each of the families
(\ref{b2}) and(\ref{c2}) are tested on a number of nonlinear
equations. To obtain a high accuracy and avoid the loss of
significant digits, we employed multi-precision arithmetic with
7000 significant
decimal digits in the programming package of Mathematica 8.\\

\begin{method}\label{mm1}
Weight  functions $G$ and $H$ in (\ref{b2}) are given  by
\begin{equation}\label{dd1}
\begin{split}
&G(t_n)=-6t_n^3+5t_n^2+2t_n+1,\\
&H(t_n,s_n,u_n)=1+2t_n+4u_n+6t_n^2+s_n,
\end{split}
\end{equation}

where $t_n=\frac{f(y_n)}{f(x_n)}$, $s_n=\frac{f(z_n)}{f(y_n)}$ and
$u_n=\frac{f(z_n)}{f(x_n)}$.

These functions  satisfy the given conditions in Theorems
\ref{theorem1} and \ref{theorem2}, so
\begin{equation}\label{dd2}
\begin{cases}
y_{n}=x_n-\dfrac{f(x_n)}{f'(x_n)},\\
z_{n}=y_n- (-6t_n^3+5t_n^2+2t_n+1) \cdot\dfrac{f(y_n)}{f'(x_n)},\\
x_{n+1}=z_n-(1+2t_n+4u_n+6t_n^2+s_n)\cdot\dfrac{f(z_n)}{f'(x_n)}.
\end{cases}
\end{equation}
\end{method}

\begin{method}\label{mmmm0}
The method by H.T. Kung and J.F. Traub \cite{Kung} is given by
\begin{equation}\label{KT0}
\begin{cases}
y_n=x_n-\dfrac{f(x_n)}{f'(x_n)},\\
z_n=y_n-G_f(x_n),\\
x_{n+1}=z_n-f^2(x_n)f(y_n)H_f(x_n,y_n,z_n),\\
\end{cases}
\end{equation}
where\\
\\
\begin{equation*}
\begin{split}
G_f(x_n)&=\frac{f^2(x_n)f(y_n)}{f'(x_n)(f(x_n)-f(y_n))^2},\\
H_f(x_n,y_n,z_n)&=G_f(x_n)\\
&\big(\frac{-1}{f^2(x_n)(f(x_n)-f(z_n))}+\frac{f(y_n)-f(x_n)}{f(x_n)f(y_n)(f(x_n)-f(z_n))^2}\\
&+\frac{1}{(f(y_n)-f(z_n))(f(x_n)-f(z_n))^2}\big).
\end{split}
\end{equation*}
\end{method}

\begin{method}\label{mmmm11}
The method by B. Neta \cite{Neta0} is given by
\begin{equation}\label{NNN}
\begin{cases}
y_n=x_n-\dfrac{f(x_n)}{f'(x_n)},\\
z_n=y_n-\dfrac{f(x_n)+Af(y_n)}{f(x_n)+(A-2)f(y_n)}\dfrac{f(y_n)}{f'(x_n)},
\quad A\in \mathbb{R},\\
x_{n+1}=y_n+\delta_1f^2(x_n)+\delta_2f^3(x_n),
\end{cases}
\end{equation}
where\\
\\
\begin{equation*}
F_y=f(y_n)-f(x_n), \quad F_z=f(z_n)-f(x_n),
\end{equation*}
\begin{equation*}
 \zeta_y=\dfrac{1}{F_y}
\left(\dfrac{y_n-x_n}{F_y}-\dfrac{1}{f'(x_n)}\right),
\zeta_z=\dfrac{1}{F_z}
\left(\dfrac{z_n-x_n}{F_z}-\dfrac{1}{f'(x_n)}\right),
\end{equation*}
\begin{equation*}
\delta_1=\zeta_y+\delta_2F_y, \quad \quad
\delta_2=-\dfrac{\zeta_y-\zeta_z}{F_y-F_z},
\end{equation*}
\end{method}

\begin{method}
The method by Khattri and Steihaug \cite{Khattri} is given by
\begin{equation}\label{kh1}
\begin{cases}
y_n=x_n-\frac{f(x_n)}{f'(x_n)},\\
z_n=y_n-\frac{f(y_n)}{\frac{x_n-y_n+\alpha f(x_n)}{(x_n-y_n)\alpha}-\frac{(x_n-y_n)f(x_n+\alpha f(x_n))}{(x_n-y_n+\alpha f(x_n))\alpha f(x_n)}-\frac{(2x_n-2y_n+\alpha f(x_n))f(y_n)}{(x_n-y_n)(x_n-y_n+\alpha f(x_n))}},\quad \alpha\in \mathbb{R},\\
x_{n+1}=z_n-\frac{f(z_n)}{H_1f(x_n)+H_2f(x_n+\alpha
f(x_n))+H_3f(y_n)+H_4f(z_n)},
\end{cases}
\end{equation}
where
\begin{equation*}
\begin{split}
H_1&=\frac{(y_n-z_n)(z_n-x_n-\alpha f(x_n))}{\alpha
f(x_n)+(y_n-x_n)(z_n-x_n)},\\
H_2&=\frac{-x_ny_n+x_nz_n+y_nz_n-z_n^2}{\alpha f(x_n)(-\alpha
f(x_n)+y_n-x_n)(\alpha f(x_n)+x_n-z_n)},\\
H_3&=-\frac{(x_n+\alpha f(x_n))x_n-(x_n+\alpha f(x_n))z_n-x_n
z_n+z_n^2}{(\alpha f(x_n)+x_n-y_n)(y_n-x_n)(y_n-z_n)},\\
H_4&=-\frac{(x_n+\alpha f(x_n))x_n+(x_n+\alpha f(x_n))y_n+x_n
y_n-2(x_n+\alpha f(x_n))z_n-2x_nz_n-2y_nz_n+3z_n^2}{(\alpha
f(x_n)+x_n-z_n)(z_n-x_n)(z_n-y_n)}.
\end{split}
\end{equation*}
\end{method}

\begin{method}\label{m1}
Weight  functions $G$, $H$, $I$, $J$, $K$, $L$, $M$ and $N$ in
(\ref{c2}) are given  by
\begin{equation}\label{d1}
\begin{split}
&G(t_n)=-6t_n^3+5t_n^2+2t_n+1,\\
&H(t_n,s_n,u_n)=1+2t_n+4u_n+6t_n^2+s_n,\\
&I(t_n)=6t_n^2+2t_n, \quad J(s_n)=-s_n^3+s_n+1, \quad K(u_n)=4u_n-4u_n^2,\\
&L(t_n, u_n)=t_nu_n+6t_n^2u_n+2t_n^3u_n-10t_nu_n^2, \quad M(p_n,q_n,r_n)=r_n+2q_n+8p_n,\\
&N(t_n,s_n,u_n,r_n)=2t_nr_n+2s_nu_n+6t_n^2r_n-4s_n^2u_n+24t_n^4u_n,\\
\end{split}
\end{equation}

where $t_n=\frac{f(y_n)}{f(x_n)}$, $s_n=\frac{f(z_n)}{f(y_n)}$,
$u_n=\frac{f(z_n)}{f(x_n)}$, $p_n=\frac{f(w_n)}{f(x_n)}$,
$q_n=\frac{f(w_n)}{f(y_n)}$ and $r_n=\frac{f(w_n)}{f(z_n)}$.

These functions  satisfy the given conditions in Theorems
\ref{theorem1}, \ref{theorem2}, and \ref{theorem3}, so
\begin{equation}\label{d2}
\begin{cases}
y_{n}=x_n-\dfrac{f(x_n)}{f'(x_n)},\\
z_{n}=y_n- (-6t_n^3+5t_n^2+2t_n+1) \cdot\dfrac{f(y_n)}{f'(x_n)},\\
w_{n}=z_n-(1+2t_n+4u_n+6t_n^2+s_n)\cdot\dfrac{f(z_n)}{f'(x_n)},\\
x_{n+1}=w_n-\big[1+6t_n^2+2t_n-s_n^3+s_n+4u_n-4u_n^2\\
\quad \quad +t_nu_n+6t_n^2u_n+2t_n^3u_n-10t_nu_n^2+r_n+2q_n+8p_n\\
\quad \quad +2t_nr_n+2s_nu_n+6t_n^2r_n-4s_n^2u_n+24t_n^4u_n
\big]\cdot\dfrac{f(w_n)}{f'(x_n)}.
\end{cases}
\end{equation}
\end{method}
\begin{method}\label{m2} Weight  functions $G$, $H$, $I$, $J$, $K$, $L$, $M$, and $N$ in (\ref{c2}) are given  by
\begin{equation}\label{d3}
\begin{split}
&G(t_n)=t_n^2(5-7t_n)+(2t_n+1)(t_n^3+1)-2t_n^4,\\
&H(t_n,s_n,u_n)=(1+s_n)+(6+u_n^2)(u_n+t_n^2)+2(t_n-u_n),\\
&I(t_n)=(1+t_n)(2t_n+t_n^2)+t_n^2(3-t_n), \quad J(s_n)=\frac{s_n+s_n^2-s_n^3}{1+s_n}, \quad K(u_n)=\frac{1+5u_n}{1+u_n},\\
&L(t_n, u_n)=t_nu_n+6t_n^2u_n+\frac{2t_n^3u_n-10t_nu_n^2}{1+t_nu_n}, \quad M(p_n,q_n,r_n)=2(p_n+q_n)+\frac{6p_n+r_n}{1+p_n},\\
&N(t_n,s_n,u_n,r_n)=8t_n^2r_n-4s_n^2u_n-2t_n^3r_n+\frac{2s_nu_n+2t_nr_n+24t_n^4u_n+2t_ns_nu_n}{1+t_n},\\
\end{split}
\end{equation}
where $t_n=\frac{f(y_n)}{f(x_n)}$, $s_n=\frac{f(z_n)}{f(y_n)}$,
$u_n=\frac{f(z_n)}{f(x_n)}$, $p_n=\frac{f(w_n)}{f(x_n)}$,
$q_n=\frac{f(w_n)}{f(y_n)}$ and $r_n=\frac{f(w_n)}{f(z_n)}$.

These functions  satisfy the given conditions in Theorems
\ref{theorem1}, \ref{theorem2}, and \ref{theorem3}, so
\begin{equation}\label{d4}
\begin{cases}
y_n=x_n-\dfrac{f(x_n)}{f'(x_n)},\\
z_{n}=y_n- (t_n^2(5-7t_n)+(2t_n+1)(t_n^3+1)-2t_n^4) \cdot\dfrac{f(y_n)}{f'(x_n)},\\
w_{n}=z_n-\left((1+s_n)+(6+u_n^2)(u_n+t_n^2)+2(t_n-u_n)\right)\cdot\dfrac{f(z_n)}{f'(x_n)},\\
x_{n+1}=w_n-\big[(1+t_n)(2t_n+t_n^2)+3t_n^2-t_n^3+8t_n^2r_n-4s_n^2u_n-2t_n^3r_n+t_nu_n+6t_n^2u_n+2(p_n+q_n)\\
\quad
\quad+\frac{1+5u_n}{1+u_n}+\frac{2t_n^3u_n-10t_nu_n^2}{1+t_nu_n}+\frac{6p_n+r_n}{1+p_n}+\frac{s_n+s_n^2-s_n^3}{1+s_n}+\frac{2s_nu_n+2t_nr_n+24t_n^4u_n+2t_ns_nu_n}{1+t_n}\big]\cdot\dfrac{f(w_n)}{f'(x_n)}.
\end{cases}
\end{equation}
\end{method}
\begin{method}\label{m3}
Weight  functions $G$, $H$, $I$, $J$, $K$, $L$, $M$, and $N$ in
(\ref{c2}) are given  by
\begin{equation}\label{d5}
\begin{split}
&G(t_n)=(1+t_n^2)(1+2t_n+2t_n^2)+t^2(2-8t_n-2t_n^2),\\
&H(t_n,s_n,u_n)=4u_n-5s_n+(6+s_n^3)(t_n^2+s_n)+(1+u_n^3)(1+2t_n),\\
&I(t_n)=(1+t_n)(2t_n+t_n^3)+t_n^2(4-t_n-t_n^2), \quad J(s_n)=-2s_n^2+\frac{s_n+2s_n^2}{1+s_n^2},\\
&K(u_n)=1+6u_n-\frac{2u_n+6u_n^2}{1+u_n}, \quad L(t_n, u_n)=t_nu_n+\frac{2t_n^3u_n-10t_nu_n^2+6t_n^2u_n}{1+2t_nu_n},\\
&M(p_n,q_n,r_n)=\frac{1+2p_n+2q_n}{1-r_n}+\frac{6p_n}{1+q_n}-1,\\
&N(t_n,s_n,u_n,r_n)=2t_nr_n+2s_nu_n+24t_n^4u_n+\frac{6t_n^2r_n+6t_n^3r_n-4s_n^2u_n}{1+t_n},\\
\end{split}
\end{equation}
where $t_n=\frac{f(y_n)}{f(x_n)}$, $s_n=\frac{f(z_n)}{f(y_n)}$,
$u_n=\frac{f(z_n)}{f(x_n)}$, $p_n=\frac{f(w_n)}{f(x_n)}$,
$q_n=\frac{f(w_n)}{f(y_n)}$ and $r_n=\frac{f(w_n)}{f(z_n)}$.

These functions  satisfy the given conditions in Theorems
\ref{theorem1}, \ref{theorem2}, and \ref{theorem3}, so
\begin{equation}\label{d6}
\begin{cases}
y_n=x_n-\dfrac{f(x_n)}{f'(x_n)},\\
z_{n}=y_n- ((1+t_n^2)(1+2t_n+2t_n^2)+t^2(2-8t_n-2t_n^2)) \cdot\dfrac{f(y_n)}{f'(x_n)},\\
w_{n}=z_n-\left(4u_n-5s_n+(6+s_n^3)(t_n^2+s_n)+(1+u_n^3)(1+2t_n)\right)\cdot\dfrac{f(z_n)}{f'(x_n)},\\
x_{n+1}=w_n-\big[(1+t_n)(2t_n+t_n^3)+4t_n^2-t_n^3-t_n^4-2s_n^2+6u_n+2t_nr_n+2s_nu_n+24t_n^4u_n+t_nu_n\\
\quad \quad
+\frac{2t_n^3u_n-10t_nu_n^2+6t_n^2u_n}{1+2t_nu_n}+\frac{1+2p_n+2q_n}{1-r_n}+\frac{6p_n}{1+q_n}-\frac{2u_n+6u_n^2}{1+u_n}+\frac{s_n+2s_n^2}{1+s_n^2}+\frac{6t_n^2r_n+6t_n^3r_n-4s_n^2u_n}{1+t_n}\big]\cdot\dfrac{f(w_n)}{f'(x_n)}.
\end{cases}
\end{equation}
\end{method}

\begin{method}\label{mmm0}
The method by H.T. Kung and J.F. Traub \cite{Kung} is given by
\begin{equation}\label{KT}
\begin{cases}
y_n=x_n-\dfrac{f(x_n)}{f'(x_n)},\\
z_n=y_n-G_f(x_n),\\
s_n=z_n-f^2(x_n)f(y_n)H_f(x_n,y_n,z_n),\\
x_{n+1}=s_n+\frac{f^2(x_n)f(y_n)f(z_n)}{f(x_n)-f(s_n)}\left(H_f(x_n,y_n,z_n)-\frac{K_f(x_n,y_n,z_n)-L_f(x_n,y_n,z_n,s_n)}{f(x_n)-f(s_n)}\right),
\end{cases}
\end{equation}
where\\
\\
\begin{equation*}
\begin{split}
G_f(x_n)&=\frac{f^2(x_n)f(y_n)}{f'(x_n)(f(x_n)-f(y_n))^2},\\
H_f(x_n,y_n,z_n)&=G_f(x_n)\\
&\big(\frac{-1}{f^2(x_n)(f(x_n)-f(z_n))}+\frac{f(y_n)-f(x_n)}{f(x_n)f(y_n)(f(x_n)-f(z_n))^2}\\
&+\frac{1}{(f(y_n)-f(z_n))(f(x_n)-f(z_n))^2}\big),\\
K_f(x_n,y_n,z_n)&=\frac{f(x_n)(f(y_n)-f(z_n))(f(x_n)-f(y_n))-f^2(x_n)f(y_n)}{f'(x_n)(f(x_n)-f(z_n))(f(x_n)-f(y_n))^2(f(y_n)-f(z_n))},\\
L_f(x_n,y_n,z_n,s_n)&=\frac{G_f(x_n)(f(z_n)-f(s_n))-\left(f(y_n)f^2(x_n)H_f(x_n,y_n,z_n)\right)(f(y_n)-f(z_n))}{(f(y_n)-f(s_n))(f(y_n)-f(z_n))(f(z_n)-f(s_n))}.
\end{split}
\end{equation*}
\end{method}

\begin{method}\label{mmm11}
The method by B. Neta \cite{Neta0} is given by
\begin{equation}\label{NNNN}
\begin{cases}
y_n=x_n-\dfrac{f(x_n)}{f'(x_n)},\\
z_n=y_n-\dfrac{f(x_n)+Af(y_n)}{f(x_n)+(A-2)f(y_n)}\dfrac{f(y_n)}{f'(x_n)},
\quad A\in \mathbb{R},\\
s_n=y_n+\delta_1f^2(x_n)+\delta_2f^3(x_n),\\
x_{n+1}=y_n+\theta_1f^2(x_n)+\theta_2f^3(x_n)+\theta_3f^4(x_n),
\end{cases}
\end{equation}
where\\
\\
\begin{equation*}
F_y=f(y_n)-f(x_n), \quad F_z=f(z_n)-f(x_n), \quad
F_s=f(s_n)-f(x_n),
\end{equation*}
\begin{equation*}
 \zeta_y=\dfrac{1}{F_y}
\left(\dfrac{y_n-x_n}{F_y}-\dfrac{1}{f'(x_n)}\right),
\zeta_z=\dfrac{1}{F_z}
\left(\dfrac{z_n-x_n}{F_z}-\dfrac{1}{f'(x_n)}\right),
\zeta_s=\dfrac{1}{F_s}
\left(\dfrac{s_n-x_n}{F_s}-\dfrac{1}{f'(x_n)}\right),
\end{equation*}
\begin{equation*}
\delta_1=\zeta_y+\delta_2F_y, \quad \quad
 \delta_2=-\dfrac{\zeta_y-\zeta_z}{F_y-F_z},\quad \quad \gamma_1=\dfrac{\zeta_s-\zeta_z}{F_s-F_z},\quad \quad
\gamma_2=\dfrac{\zeta_y-\zeta_z}{F_y-F_z},
 \end{equation*}
\begin{equation*}
\theta_1=\zeta_s+\theta_2F_s-\theta_3F_s^2, \quad
\theta_2=\gamma_1+\theta_3(F_s-+F_z), \quad
\theta_3=\dfrac{\gamma_1-\gamma_2}{F_s-F_y},
\end{equation*}
\end{method}

\begin{method}\label{m4}
The method by Y. H. Geum and Y. I. Kim \cite{Geum1} is given by
\begin{equation}\label{d7}
\begin{cases}
y_n=x_n-\dfrac{f(x_n)}{f'(x_n)},\\
z_{n}=y_n- K_f(u_n) \cdot\dfrac{f(y_n)}{f'(x_n)},\\
s_{n}=z_n-H_f(u_n, v_n, w_n)\cdot\dfrac{f(z_n)}{f'(x_n)},\\
x_{n+1}=s_n-W_f(u_n,v_n,w_n,t_n)\cdot\dfrac{f(s_n)}{f'(x_n)}.
\end{cases}
\end{equation}
where
\begin{equation*}
\begin{split}
&K_f(u_n)=\dfrac{1+\beta u_n+\lambda u_n^2}{1+(\beta-2)u_n+\mu u_n^2},\\
&H_f(u_n,v_n,w_n)=\dfrac{1+a u_n +b v_n+\gamma w_n}{1+c u_n+d v_n+\sigma w_n},\\
&W_f(u_n,v_n,w_n,t_n)=\dfrac{1+B_1u_n+B_2v_nw_n}{1+B_3v_n+B_4w_n+B_5t_n+B_6v_nw_n}+G(u_n,w_n),
\end{split}
\end{equation*}
and $u_n=\frac{f(y_n)}{f(x_n)}$, $v_n=\frac{f(z_n)}{f(y_n)}$,
$w_n=\frac{f(z_n)}{f(x_n)}$ and $t_n=\frac{f(s_n)}{f(z_n)}$.\\
\begin{equation*}\label{d8}
\begin{split}
&a=2, \quad b=0, \quad  c=0, \quad  d=-1, \quad  \gamma=2+\sigma, \quad \lambda=-9+\frac{5\beta}{2},\quad \mu=-4+\frac{\beta}{2}, \\
&B_1=2, \quad B_2=2+\sigma, \quad B_3=-1, \quad B_4=-2, \quad
B_5=-1, \quad B_6=2(1+\sigma),
\end{split}
\end{equation*}
With weight function
\begin{equation}\label{d8}
\begin{split}
G(u_n,w_n)&=\frac{-1}{2}\left(u_nw_n\left(6+12u_n+u_n^2(24-11\beta)+u_n^3\phi_1+4\sigma\right)\right)+\phi_2w_n^2,\\
&\beta=2, \quad \sigma=-2, \quad \phi_1=11\beta^2-66\beta+136,
\quad \phi_2=2u_n(\sigma^2-2\sigma-9)-4\sigma-6.\\
\end{split}
\end{equation}

\end{method}
\begin{method}\label{m5}
The method by Y. H. Geum and Y. I. Kim \cite{Geum2} is given by
\begin{equation}\label{d9}
\begin{cases}
y_n=x_n-\dfrac{f(x_n)}{f'(x_n)},\\
z_{n}=y_n- K_f(u_n) \cdot\dfrac{f(y_n)}{f'(x_n)},\\
s_{n}=z_n-H_f(u_n, v_n, w_n)\cdot\dfrac{f(z_n)}{f'(x_n)},\\
x_{n+1}=s_n-W_f(u_n,v_n,w_n,t_n)\cdot\dfrac{f(s_n)}{f'(x_n)}.
\end{cases}
\end{equation}
where
\begin{equation*}
\begin{split}
&K_f(u_n)=\dfrac{1+\beta u_n+\gamma u_n^2}{1+(\beta-2)u_n+\mu u_n^2},\\
&H_f(u_n,v_n,w_n)=\dfrac{1+a u_n+b v_n +\gamma w_n}{1+c u_n+dv_n+\sigma w_n},\\
&W_f(u_n,v_n,w_n,t_n)=\dfrac{1+A_1u_n}{1+A_2v_n+A_3w_n+A_4t_n}+G(u_n,v_n,w_n),
\end{split}
\end{equation*}
also $u_n=\frac{f(y_n)}{f(x_n)}$, $v_n=\frac{f(z_n)}{f(y_n)}$,
$w_n=\frac{f(z_n)}{f(x_n)}$ and $t_n=\frac{f(s_n)}{f(z_n)}$,
\begin{equation*}
\begin{split}
 &a=2, \quad b=0, \quad c=0, \quad
d=-1,\\
&\gamma=2+\sigma, \quad \lambda=-9+\frac{5\beta}{2}, \quad
\mu=-4+\frac{\beta}{2},\\
&A_1=2, \quad A_2=-1, \quad A_3=-2, \quad A_4=-1,
\end{split}
\end{equation*}
with the weight function
\begin{equation}\label{d10}
\begin{split}
G(u_n,v_n,w_n)&=-6u_n^3v_n+6w_n^2-4u_n^4(3v_n+17w_n)+u_n(2v_n^2+4v_n^3+w_n-2w_n^2),\\
\beta&=0, \quad \sigma=-2.
\end{split}
\end{equation}

\end{method}

We test our proposed methods  \eqref{d2}, \eqref{d4}, and
\eqref{d6} on the functions $f_1, \dots, f_6$ described in Table
\ref{table1}, the Table also lists the exact roots $x^*$ and
initial approximations $x_0$, which are computed using the
\texttt{FindRoot} command of Mathematica \cite[pp.\
158--160]{Hazrat}.

For the methods defined by \eqref{d2}, \eqref{d4}, and \eqref{d6},
we display in Tables \ref{table2}-\ref{table4} for the weight
functions \eqref{d1}, \eqref{d3}, \eqref{d5} the errors $|x_n -
x^*|$ for $n=1,2,3$, and the \emph{computational order of
convergence (coc)} \cite{Fer}, approximated by
\begin{equation*}
  \mathrm{coc} \approx \frac{\ln|(x_{n+1}-x^{*})/(x_{n}-x^{*})|}{\ln|(x_{n}-x^{*})/(x_{n-1}-x^{*})|}.
\end{equation*}

\begin{table}[!ht]
\begin{center}
\begin{tabular}{l c c}
  \hline
    Test function $f_n$ & root $x^{*}$ & initial approximation $x_0$
  \\ \hline
    $f_1(x) = \ln (1+x^2)+e^{x}\sin x$ & $0$ & $0.03$
  \\[0.5ex]
   $f_2(x) = \frac{-x}{100}+\sin x$ & $0$ & $0.5$
  \\[0.5ex]
   $ f_3(x)=x \ln (1+x \sin x)+ e^{-1+x^2+ x \cos x}\sin \pi x $ & $ 0 $ & $ 0.01 $
  \\[0.5ex]
 $ f_4(x)=1+e^{2+x-x^2}+x^3-\cos(1+x) $& $ -1 $ & $ -0.3 $
  \\[0.5ex]
  $ f_5(x)=(1-\sin x^2)\frac{x^2+1}{x^3+1}+x\ln(x^2-\pi+1)-\frac{1+\pi}{1+\sqrt{\pi^3}}$&$\sqrt{\pi}$&$1.7$
  \\[0.5ex]
 $ f_6(x)=(1+x^2)\cos(\frac{\pi x}{2})+\frac{\ln(x^2+2x+2)}{1+x^2}$&$-1$&$-1.1$
  \\ \hline
\end{tabular}
\end{center}
\vspace*{-3ex}
\caption{Test functions $f_1, \dots, f_6$, root $x^*$ and initial approximation $x_0$.\label{table1}}
\end{table}

To obtain a high accuracy and avoid loss of significant digits, we
employ multi-precision arithmetic with $6000$ significant decimal
digits in Mathematica 8.

\begin{table}[!ht]
\begin{center}
\begin{tabular}{c c c c c }
\hline $f_n$ & $~~~~\vert x_{1} - x^{*} \vert~~~~$ & $~~~~\vert x_{2} - x^{*} \vert~~~~$ & $~~~~\vert x_{3} - x^{*} \vert~~~~$ & coc\\
\hline
$f_1$ & $0.380e-20$ & $0.126e-319$ & $0.276e-5111$ & $16.0000$ \\
$f_2$ & $0.104e-10$ & $0.265e-192$ & $0.211e-3279$ & $17.0000$ \\
$f_3$ & $0.450e-28$& $0.303e-449$ & $0.561e-7188$ & $16.0000$ \\
$f_4$ & $0.609e-8$ & $0.465e-136$ & $0.630e-2186$ & $16.0000$\\
$f_5$ & $0.246e-14$ & $0.276e-230$ & $0.169e-3685$ & $16.0000$\\
$f_6$ & $0.142e-17$ & $0.482e-283$ & $0.139e-4530$ & $16.0000$\\
\hline
\end{tabular}
\end{center}
\vspace*{-3ex} \caption{Errors and coc for method \eqref{d2}.
\label{table2}}
\end{table}

\hspace{0.5cm}
\begin{table}[!ht]
\begin{center}
\begin{tabular}{c c c c c }
\hline  $f_n$ & $\vert x_{1} - x^{*} \vert$ & $\vert x_{2} -x^{*} \vert$ & $\vert x_{3} - x^{*} \vert$ & coc
\\
\hline
$f_1$ & $0.144e-19$ & $0.193e-310$ & $0.222e-4964$ & $16.0000$ \\
$f_2$ & $0.301e-23$ & $0.339e-451$ & $0.336e-8582$ & $19.0000$ \\
$f_3$ & $0.405e-28$ & $0.515e-450$ & $0.239e-7200$ & $16.0000$ \\
$f_4$ & $0.628e-8$ & $0.276e-135$ & $0.561e-2173$ & $16.0000$\\
$f_5$ & $0.224e-14$ & $0.526e-231$ & $0.456e-3697$ & $16.0000$\\
$f_6$ & $0.186e-17$ & $0.322e-281$ & $0.202e-4501$ & $16.0000$\\
\hline
\end{tabular}
\end{center}
\vspace*{-3ex} \caption{Errors and coc for method \eqref{d4}.
\label{table3}}
\end{table}
\hspace{0.5cm}
\begin{table}[h!]
\begin{center}
\begin{tabular}{ c c c c c }
\hline $f_n$ & $\vert x_{1} - x^{*} \vert$ & $\vert x_{2} - x^{*} \vert$ & $\vert x_{3} - x^{*} \vert$ & coc
\\
\hline
$f_1$ & $0.389e-20$ & $0.931e-321$ & $0.107e-5130$ & $16.0000$ \\
$f_2$ & $0.414e-22$ & $0.1220e-385$ & $0.117e-6565$ & $17.0000$ \\
$f_3$ & $0.936e-29$ & $0.865e-461$ & $0.243e-7373$ & $16.0000$ \\
$f_4$ & $0.554e-8$ & $0.756e-136$ & $0.108e-2181$ & $16.0000$\\
$f_5$ & $0.119e-14$ & $0.116e-235$ & $0.760e-3772$ & $16.0000$\\
$f_6$ & $0.125e-17$ & $0.926e-285$ & $0.745e-4559$ & $16.0000$\\
\hline
\end{tabular}
\end{center}
\vspace*{-3ex} \caption{Errors and coc for method \eqref{d6}
.\label{table4}}
\end{table}

Tables \ref{table2}-\ref{table4} show that methods \eqref{d2},
\eqref{d4}, and \eqref{d6} support the convergence analysis given
in the previous sections.

In Tables \ref{table5}-\ref{table7}, we compare our three-point
method \eqref{dd2} with the methods \eqref{KT0}, \eqref{NNN} and
\eqref{kh1} and our four-point methods \eqref{d2}, \eqref{d4} and
\eqref{d6} with the methods \eqref{KT}, \eqref{NNNN}, \eqref{d7}
and \eqref{d9}.

\begin{table}[!ht]
\begin{center}
\begin{tabular}{ c c c c c c} \hline
 Method & Weight function & $\vert x_{1}- x^{*}\vert$ & $\vert
x_{2} - x^{*} \vert$ & $\vert x_{3} - x^{*} \vert$ & coc
\\
\hline
$(\ref{dd2})$ &(\ref{dd1})& $~~0.834e-13$ & $0.112e-121~~$ & $0.160e-1101$ &$9.0000$\\
$(\ref{KT0})$ &-& $~~0.316e-11$ & $0.165e-95~~$ & $0.918e-770$ &$8.0000$\\
$(\ref{NNN})$ &-& $~~0.445e-11$ & $0.380e-94~~$ & $0.108e-758$ &$8.0000$\\
$(\ref{kh1})$ &-& $~~0.704e-11$ & $0.658e-92~~$ & $0.384e-740$ &$8.0000$\\
$(\ref{d2})$ &(\ref{d1})& $~~0.844e-25$ & $0.659e-435~~$ & $0.257e-7406$ &$17.0000$\\
$(\ref{d4})$ &(\ref{d3})& $~~0.102e-24$ &$0.229e-433~~$ & $0.215e-7380$& $17.0000$ \\
$(\ref{d6})$ &(\ref{d5})& $~~0.371e-25$ &$0.429e-465~~$ & $0.571e-8384$& $18.0000$ \\
$(\ref{KT})$ &-& $~~0.114e-22$ &$0.323e-374~~$ & $0.546e-5999$& $16.0000$ \\
$(\ref{NNNN})$ &-& $~~0.226e-22$ &$0.403e-369~~$ & $0.429e-5917$& $16.0000$ \\
$(\ref{d7})$ &(\ref{d8})&  $~~0.235e-22$ &$0.106e-368~~$ & $0.316e-5910$& $16.0000$ \\
$(\ref{d9})$ &(\ref{d10})& $~~0.310e-21$ &$0.600e-350~~$
&$0.226e-5609$& $16.0000$\\\hline
\end{tabular}
\end{center}
\vspace*{-3ex}
 \caption{Comparison for $f(x) =
\ln(1-x+x^2)+4\sin (1-x)$, zero $x^{*}=1$ and initial
$x_{0}=1.1$.\label{table5}}
\end{table}
\hspace{1cm}
\begin{table}[h!]
\begin{center}
\begin{tabular}{ c c c c c c} \hline
Method & Weight function & $\vert x_{1} -x^{*} \vert$ & $\vert
x_{2} - x^{*} \vert$ & $\vert x_{3} - x^{*} \vert$ & coc
\\
\hline
$(\ref{dd2})$ &(\ref{dd1})& $~~0.125e-10$ & $0.888e-85~~$ & $0.545e-678$ &$8.0000$\\
$(\ref{KT0})$ &-& $~~0.414e-9$ & $0.961e-72~~$ & $0.811e-573$ &$8.0000$\\
$(\ref{NNN})$ &-& $~~0.597e-9$ & $0.273e-70~~$ & $0.533e-561$ &$8.0000$\\
$(\ref{kh1})$ &-& $~~0.629e-9$ & $0.393e-70~~$ & $0.929e-560$ &$8.0000$\\
$(\ref{d2})$ &(\ref{d1})& $~~0.201e-9$ & $0.510e-148~~$ & $0.142e-2365$ &$16.0000$\\
$(\ref{d4})$ &(\ref{d3})& $~~0.210e-9$ &$0.807e-148~~$ & $0.183e-2362$& $16.0000$ \\
$(\ref{d6})$ &(\ref{d5})& $~~0.389e-20$ &$0.931e-321~~$ & $0.107e-5130$& $16.0000$ \\
$(\ref{KT})$ &-& $~~0.883e-18$ &$0.897e-282~~$ & $0.114e-4505$& $16.0000$ \\
$(\ref{NNNN})$ &-& $~~0.183e-17$ &$0.254e-276~~$ & $0.470e-4418$& $16.0000$ \\
$(\ref{d7})$ &(\ref{d8})&  $~~0.274e-10$ &$0.682e-162~~$ & $0.150e-2587$& $16.0000$ \\
$(\ref{d9})$ &(\ref{d10})& $~~0.115e-9$ &$0.377e-149~~$ &
$0.621e-2381$& $16.0000$\\\hline
\end{tabular}
\end{center}
\vspace*{-3ex} \caption{Comparison for $f(x) =
\ln(1+x^2)+e^{x}\sin x$, zero $x^{*}=0$ and initial
$x_{0}=0.1$.\label{table6}}
\end{table}
\hspace{1cm}
\begin{table}[h!]
\begin{center}
\begin{tabular}{ c c c c c c} \hline
Method & Weight function & $\vert x_{1} -x^{*} \vert$ & $\vert
x_{2} - x^{*} \vert$ & $\vert x_{3} - x^{*} \vert$ & coc
\\
\hline
$(\ref{dd2})$ &(\ref{dd1})& $~~0.361e-13$ & $0.209e-106~~$ & $0.270e-852$ &$8.0000$\\
$(\ref{KT0})$ &-& $~~0.230e-12$ & $0.396e-99~~$ & $0.303e-793$ &$8.0000$\\
$(\ref{NNN})$ &-& $~~0.349e-12$ & $0.168e-97~~$ & $0.485e-780$ &$8.0000$\\
$(\ref{kh1})$ &-& $~~0.477e-17$ & $0.463e-141~~$ & $0.366e-1133$ &$8.0000$\\
$(\ref{d2})$ &(\ref{d1})& $~~0.499e-25$ & $0.113e-400~~$ & $0.536e-6411$ &$16.0000$\\
$(\ref{d4})$ &(\ref{d3})& $~~0.406e-25$ &$0.345e-402~~$ & $0.256e-6435$& $16.0000$ \\
$(\ref{d6})$ &(\ref{d5})& $~~0.825e-26$ &$0.242e-414~~$ & $0.760e-6631$& $16.0000$ \\
$(\ref{KT})$ &-& $~~0.211e-24$ &$0.152e-390~~$ & $0.780e-6249$& $16.0000$ \\
$(\ref{NNNN})$ &-& $~~0.484e-24$ &$0.209e-384~~$ & $0.310e-6150$& $16.0000$ \\
$(\ref{d7})$ &(\ref{d8})&$~~0.253e-24$ &$0.173e-389~~$ &$0.372e-6232$& $16.0000$\\
$(\ref{d9})$ &(\ref{d10})& $~~0.454e-22$ &$0.108e-350~~$
&$0.123e-5608$& $16.0000$\\\hline
\end{tabular}
\end{center}
\vspace*{-3ex} \caption{Comparison for $f(x) =
\frac{-2}{27}(9\sqrt{2}+7\sqrt{3})+\sqrt{1-x^2}+(1+x^3)\cos
(\frac{\pi x}{2})$, zero $x^{*}=\frac{1}{3}$ and initial
$x_{0}=0.35$.\label{table7}}
\end{table}
\\
\\
\\
\\
\\
\\
\\
\\
It can be observed from Tables \ref{table5}-\ref{table7} that for
the presented examples our three-point proposed method \eqref{dd2}
is comparable and competitive to the methods \eqref{KT0},
\eqref{NNN} and \eqref{kh1} and our four-point proposed methods
\eqref{d2}, \eqref{d4} and \eqref{d6} are comparable and
competitive to the methods \eqref{KT}, \eqref{NNNN}, \eqref{d7}
and \eqref{d9} also.
\newpage
\subsection{Dynamic behavior}
In this section, we survey the comparison of iterative methods in
the complex plane by using basins of attraction. Studying the
dynamic behavior, using basins of attractions, of the rational
functions associated to an iterative method gives important
information about the convergence and stability of the scheme
\cite{Soleymani2}.

Neta et al.\ have compared various methods for solving nonlinear
equations with multiple roots by comparing the basins of attraction
\cite{Neta} and also Scott et al.\ have compared several methods
for approximating simple roots \cite{Scott}. Moreover, a number of iterative
root-finding methods were compared from a dynamical point of view
by Amat et al. \cite{Amat1}-\cite{Amat4}, Neta et al.
\cite{Neta1}, \cite{Neta2} Stewart \cite{Stewart}, Vrscay and
Gilbert \cite{Vrscay}. To this end, some basic concepts are
briefly recalled.

Let $G:\mathbb{C} \to \mathbb{C} $ be a rational map on the
complex plane. For $z\in \mathbb{C} $, we define its orbit as the
set $orb(z)=\{z,\,G(z),\,G^2(z),\dots\}$. A point $z_0 \in
\mathbb{C} $ is called a periodic point with minimal period $m$ if
$G^m(z_0)=z_0$, where $m$ is the smallest integer with this
property. A periodic point with minimal period $1$ is called a
fixed point. Moreover, a point $z_0$ is called attracting if
$|G'(z_0)|<1$, repelling if $|G'(z_0)|>1$, and neutral otherwise.
The Julia set of a nonlinear map $G(z)$, denoted by $J(G)$, is the
closure of the set of its repelling periodic points. The
complement of $J(G)$ is the Fatou set $F(G)$, where the basin of
attraction of the different roots lie \cite{Babajee},
\cite{Cordero5}.

We use the basin of attraction for comparing the iteration
algorithms. Approximating basins of attraction is a method to visually
comprehend how an algorithm  behaves as a function of the various
starting points.

From the dynamical point of view, in fact, we take a $256 \times
256$ grid of the square $[-3,3]\times[-3,3]\in \mathbb{C}$ and
assign a color to each grid point $z_0$ according to the simple
root to which the corresponding orbit of the iterative method
starting from $z_0$ converges, and we mark the point as black if
the orbit does not converge to a root, in the sense that after at
most 100 iterations it has a distance to any of the roots, which
is larger than $10^{-3}$. In this way, we distinguish the
attraction basins by their color for different methods.
\\
\begin{table}[!ht]
\begin{center}
\begin{tabular}{l c}
  \hline
    Test Problems $p_n$ & Roots
  \\ \hline
   $p_1(z) = z^2+1 $ & $i,\quad -i$
  \\[0.5ex]
   $p_2(z) = z^3+z $ & $0,\quad i,\quad -i$
  \\[0.5ex]
   $p_3(z) = z^3+z^2-1 $ & $ -0.877439+0.744862i, \quad -0.877439-0.744862i,\quad 0.7548878 $
    \\ \hline
\end{tabular}
\end{center}
\vspace*{-3ex} \caption{Test Problems $p_1(z), p_2(z), p_3(z)$ and
roots.\label{table8}}
\end{table}
\\
For the test problem $p_1(z)$, with its roots given in Table
\ref{table8}, the results are presented in Figures
\ref{fig:figure1}-\ref{fig:figure5}. For test problems $p_2(z)$
and $p_3(z)$, the results are shown in Figures
\ref{fig:figure6}-\ref{fig:figure10} and Figures
\ref{fig:figure11}-\ref{fig:figure15}, respectively. As a result,
the method (\ref{d4}) (see Figures \ref{fig:figure1},
\ref{fig:figure6} and \ref{fig:figure11}) seems to produce larger
basins of attraction than the methods ({\ref{d7}}) and ({\ref{d9})
(see Figures \ref{fig:figure4},
 \ref{fig:figure5}, \ref{fig:figure9}, \ref{fig:figure10}, \ref{fig:figure14} and \ref{fig:figure15}) and smaller basins of attraction than the methods (\ref{KT}) and (\ref{NNNN}) (see Figures \ref{fig:figure2},
 \ref{fig:figure3}, \ref{fig:figure7}, \ref{fig:figure8}, \ref{fig:figure12} and \ref{fig:figure13}).

Note that points might belong to no basin of attraction; these are
starting points for which the methods do not converge, approximated and visualized by
black points. These exceptional points constitute the so-called Julia set of the corresponding
methods.

\begin{figure}[ht!]
\begin{minipage}[b]{0.30\linewidth}
\centering
\includegraphics[width=\textwidth]{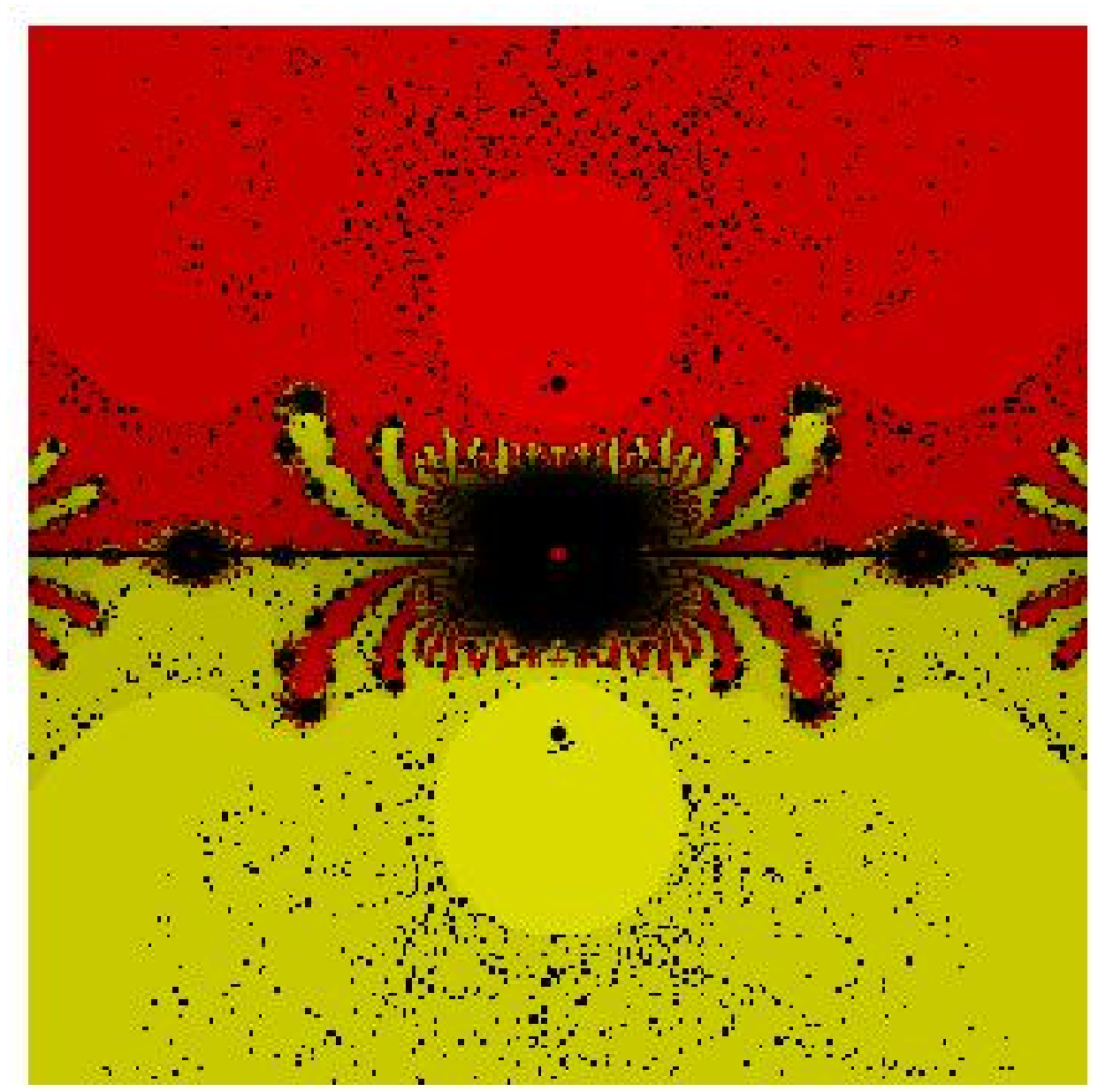}
\caption{Method (\ref{d4}) for test problem $p_1$}
\label{fig:figure1}
\end{minipage}
\hspace{0.5cm}
\begin{minipage}[b]{0.30\linewidth}
\centering
\includegraphics[width=\textwidth]{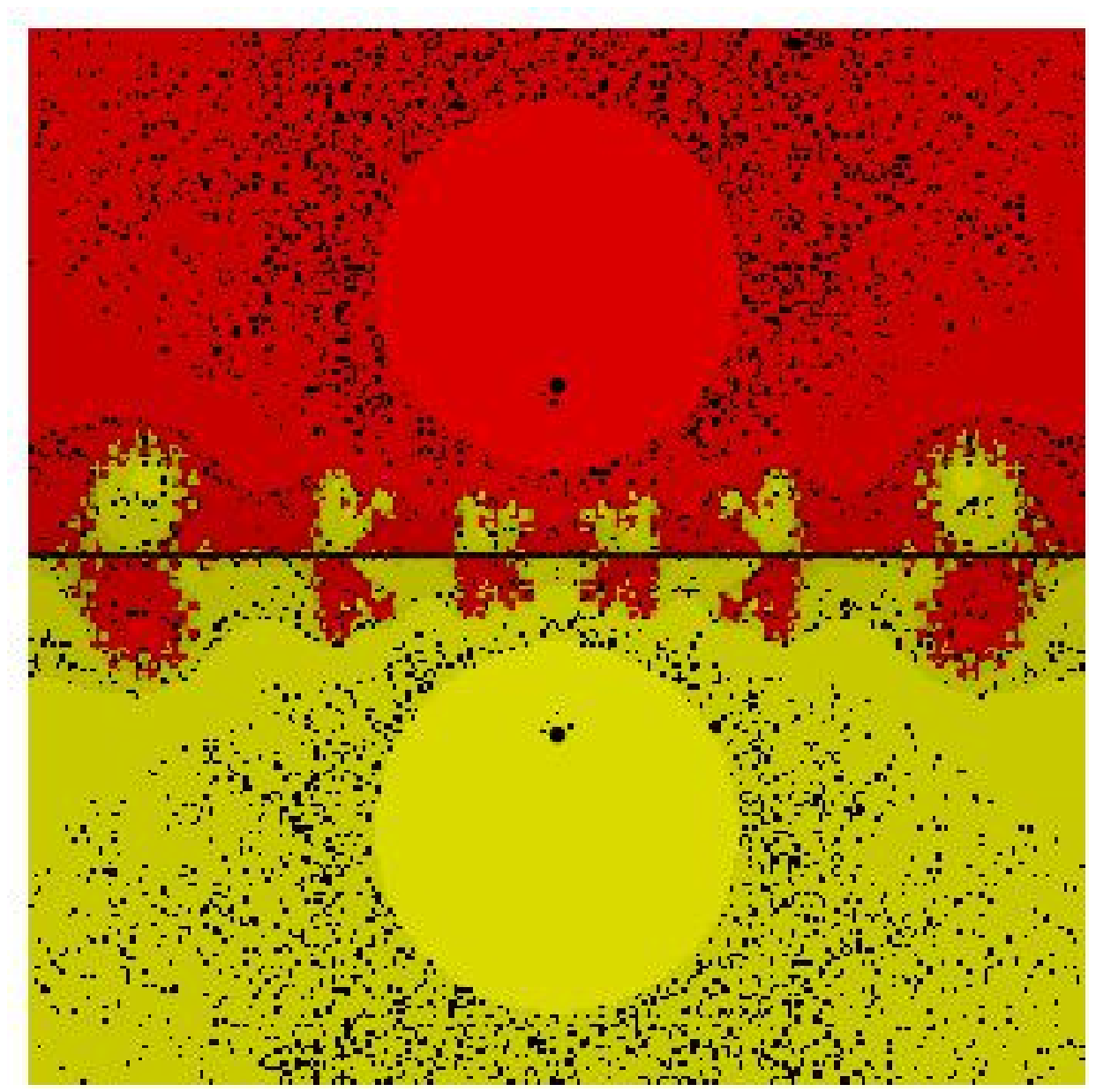}
\caption{Method (\ref{KT}) for test problem
$p_1$}\label{fig:figure2}
\end{minipage}
\hspace{0.5cm}
\begin{minipage}[b]{0.30\linewidth}
\centering
\includegraphics[width=\textwidth]{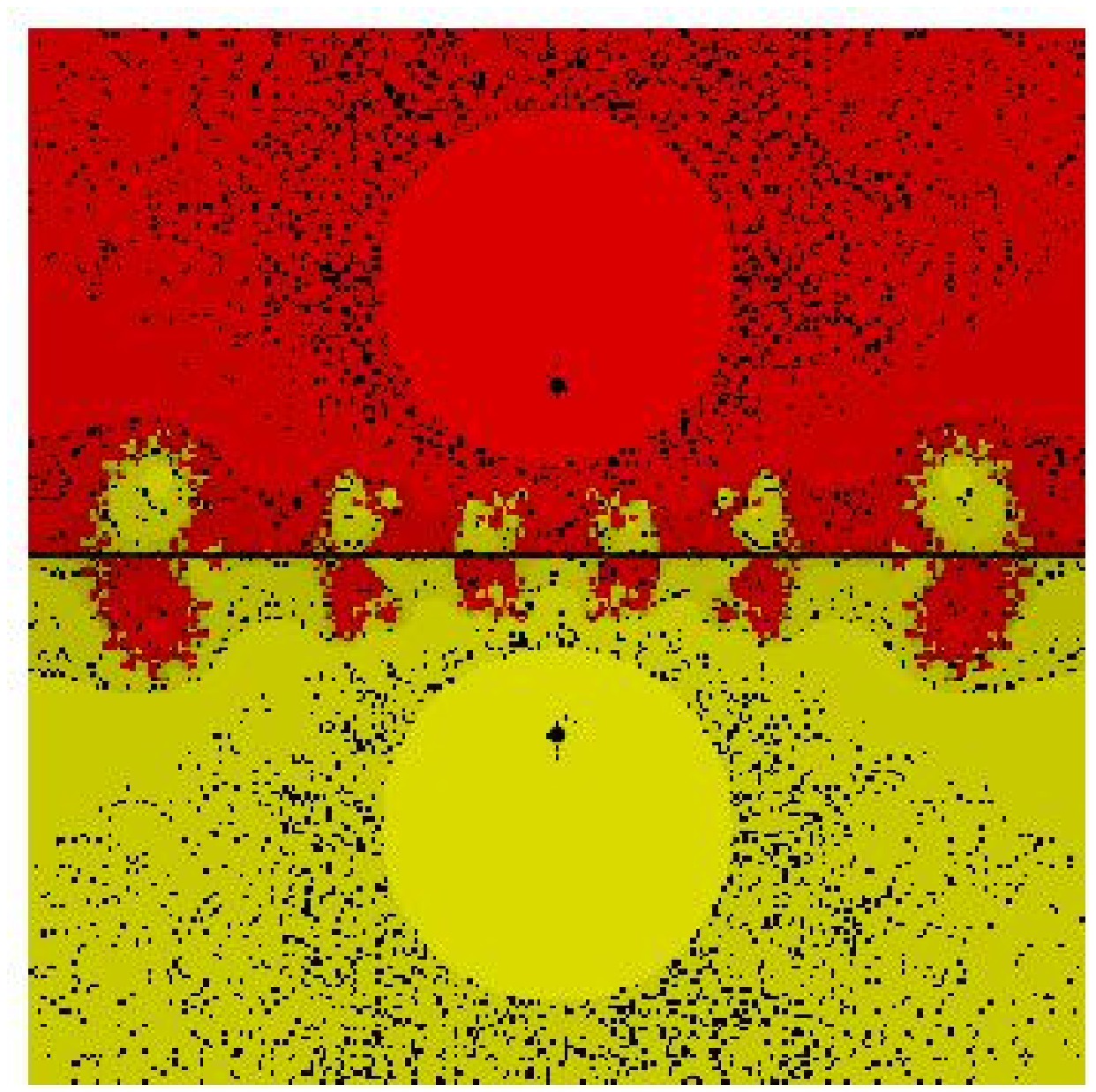}
\caption{Method (\ref{NNNN}) for test problem $p_1$}
\label{fig:figure3}
\end{minipage}
\end{figure}

\begin{figure}
\begin{minipage}[b]{0.30\linewidth} \centering
\includegraphics[width=\textwidth]{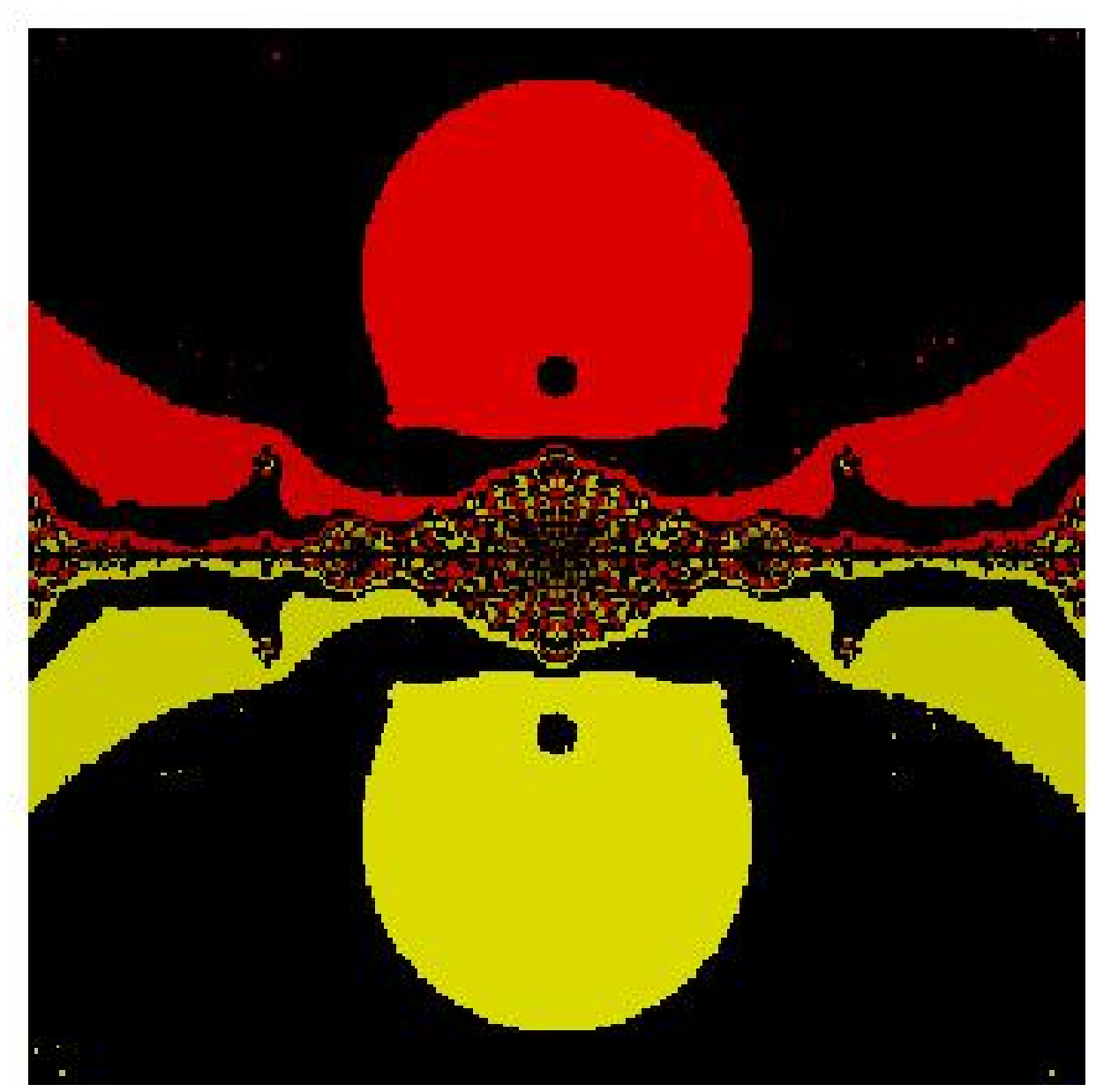}
\caption{Method (\ref{d7}) for test problem $p_1$}
\label{fig:figure4}
\end{minipage}
\hspace{2cm}
\begin{minipage}[b]{0.30\linewidth}
\centering
\includegraphics[width=\textwidth]{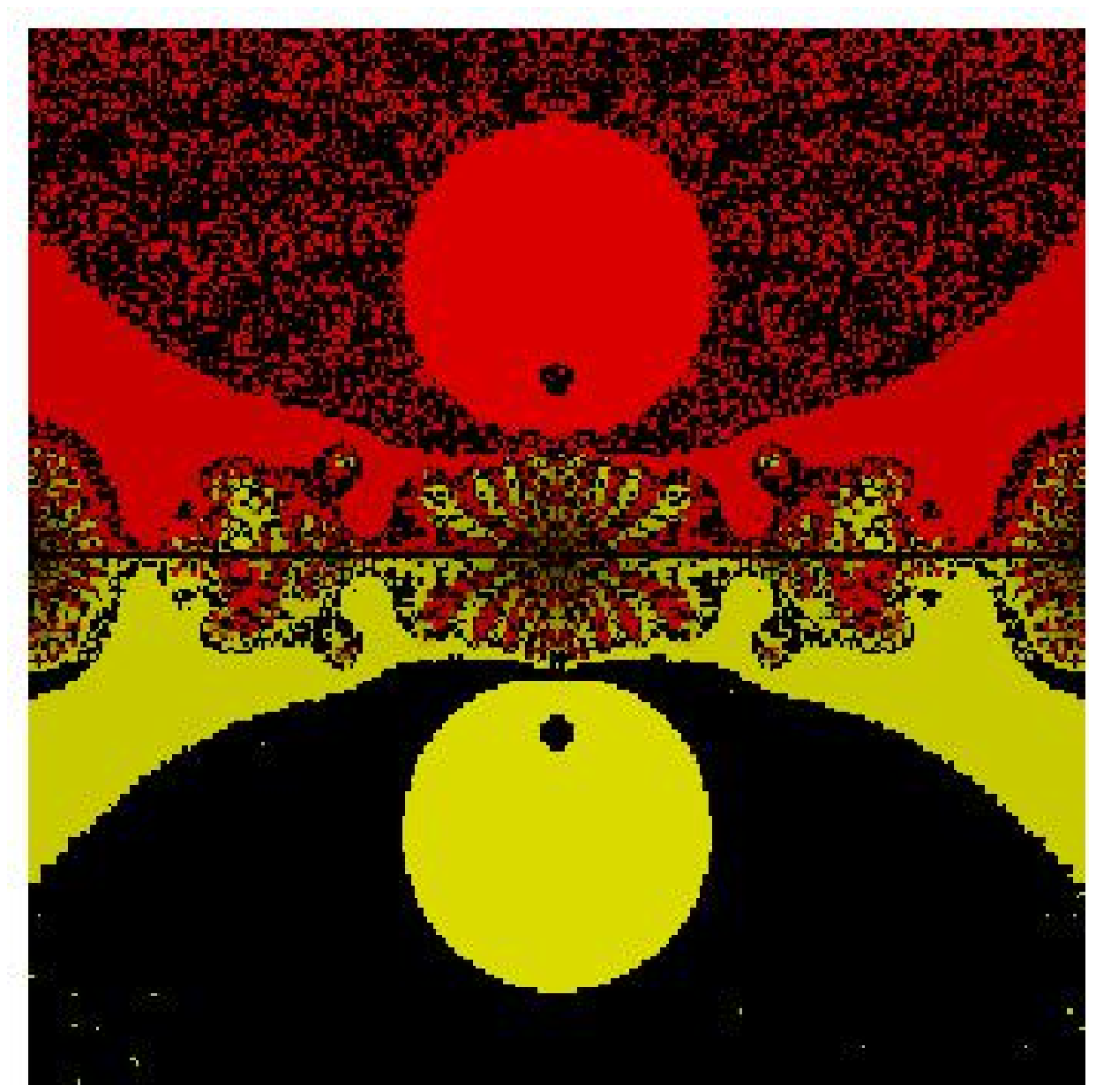}
\caption{Method (\ref{d9}) for test problem
$p_1$}\label{fig:figure5}
\end{minipage}
\end{figure}
\newpage
\begin{figure}[ht!]
\begin{minipage}[b]{0.30\linewidth}
\centering
\includegraphics[width=\textwidth]{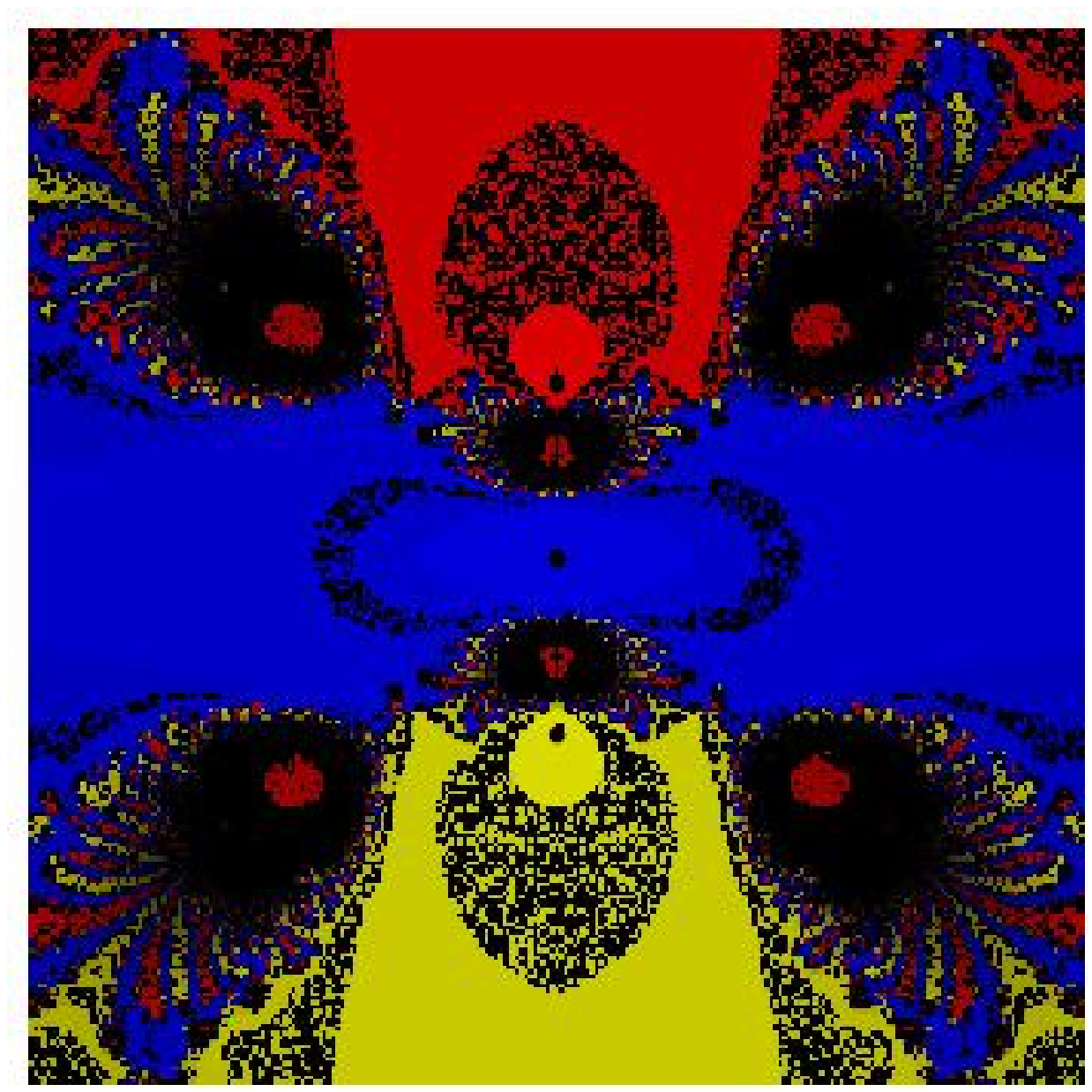}
\caption{Method (\ref{d4}) for test problem $p_2$}
\label{fig:figure6}
\end{minipage}
\hspace{0.5cm}
\begin{minipage}[b]{0.30\linewidth}
\centering
\includegraphics[width=\textwidth]{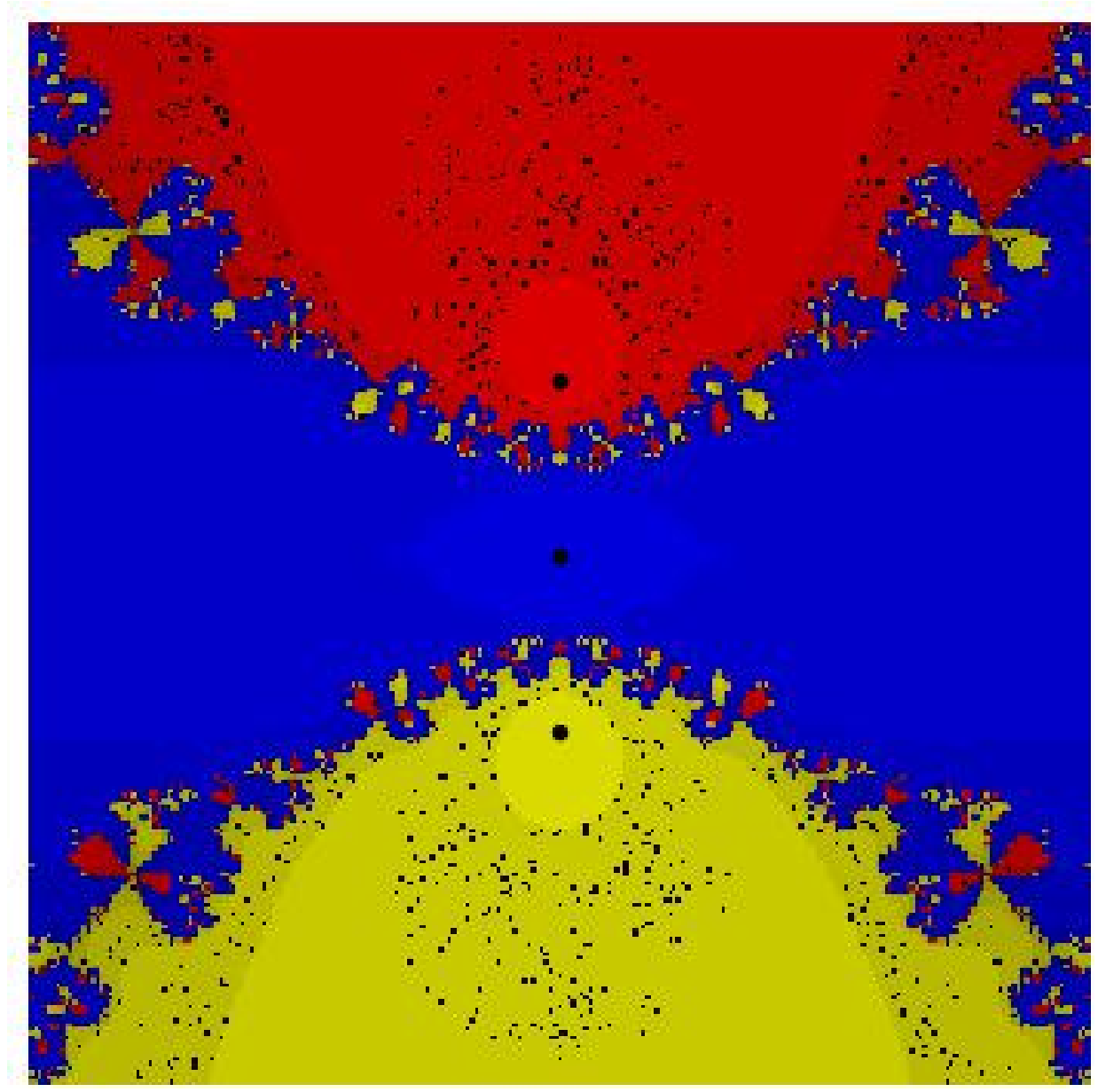}
\caption{Method (\ref{KT}) for test problem $p_2$}
\label{fig:figure7}
\end{minipage}
\hspace{0.5cm}
\begin{minipage}[b]{0.30\linewidth}
\centering
\includegraphics[width=\textwidth]{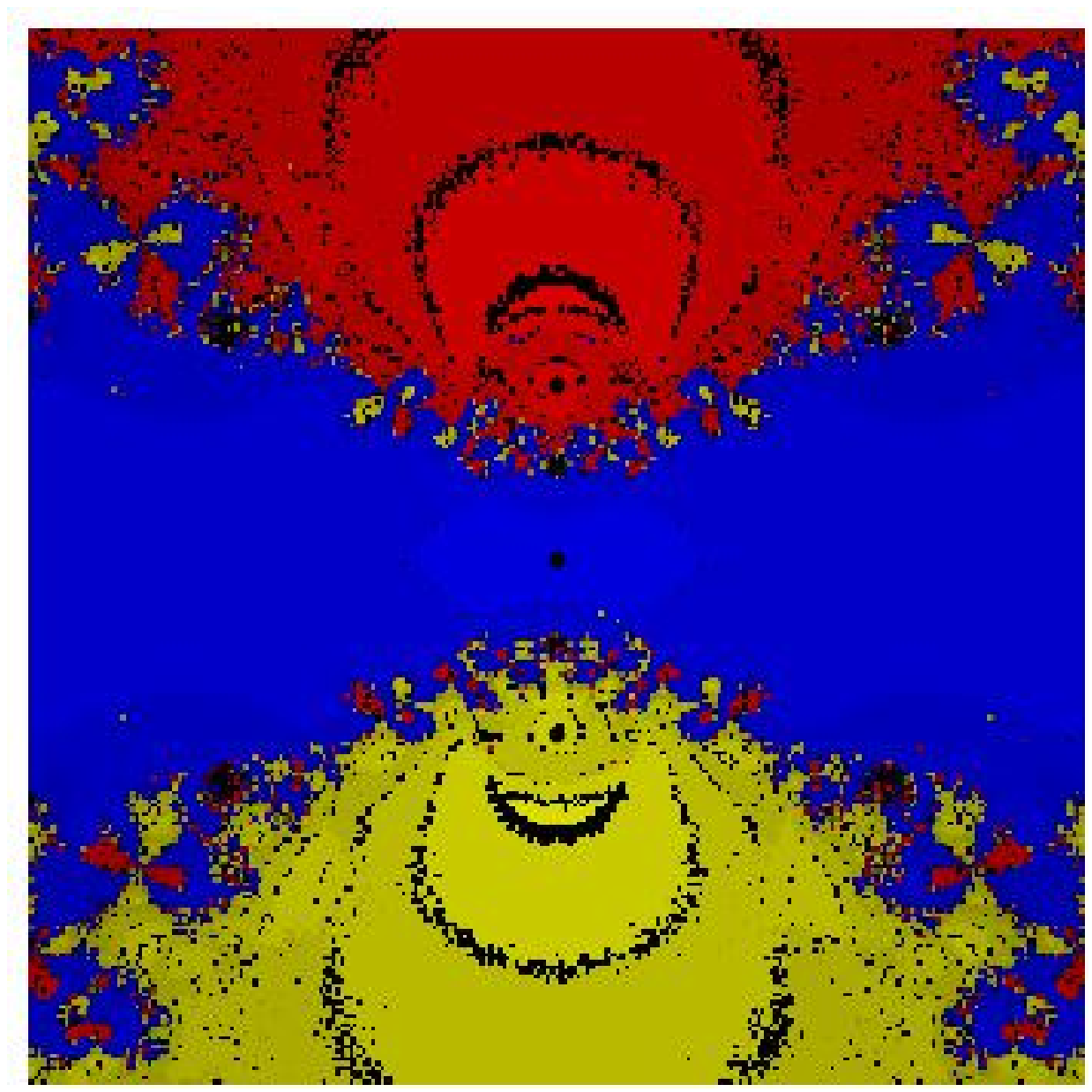}
\caption{Method (\ref{NNNN}) for test problem $p_2$}
\label{fig:figure8}
\end{minipage}
\end{figure}

\begin{figure}
\begin{minipage}[b]{0.30\linewidth}
\centering
\includegraphics[width=\textwidth]{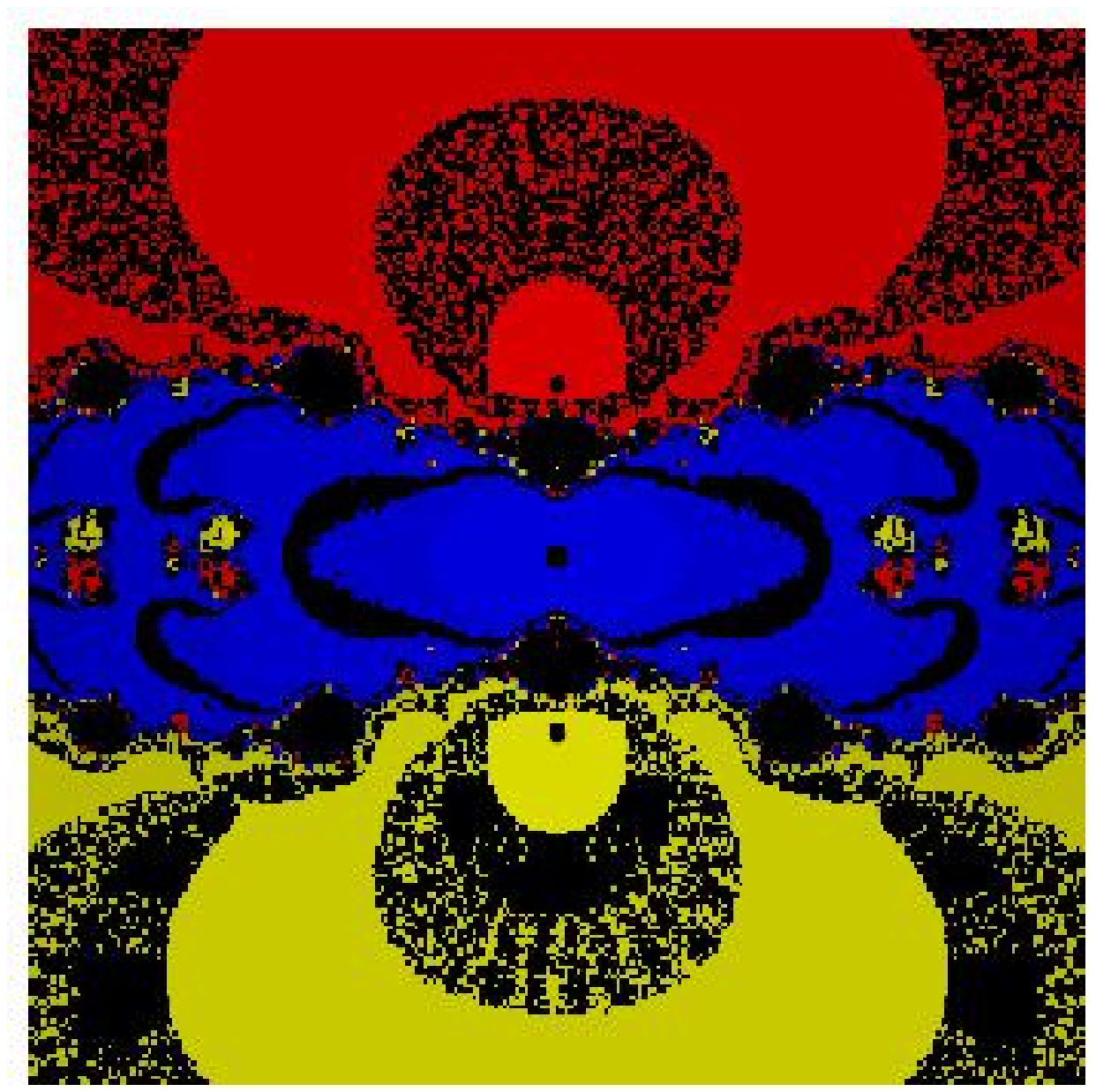}
\caption{Method (\ref{d7}) for test problem $p_2$}
\label{fig:figure9}
\end{minipage}
\hspace{2cm}
\begin{minipage}[b]{0.30\linewidth}
\centering
\includegraphics[width=\textwidth]{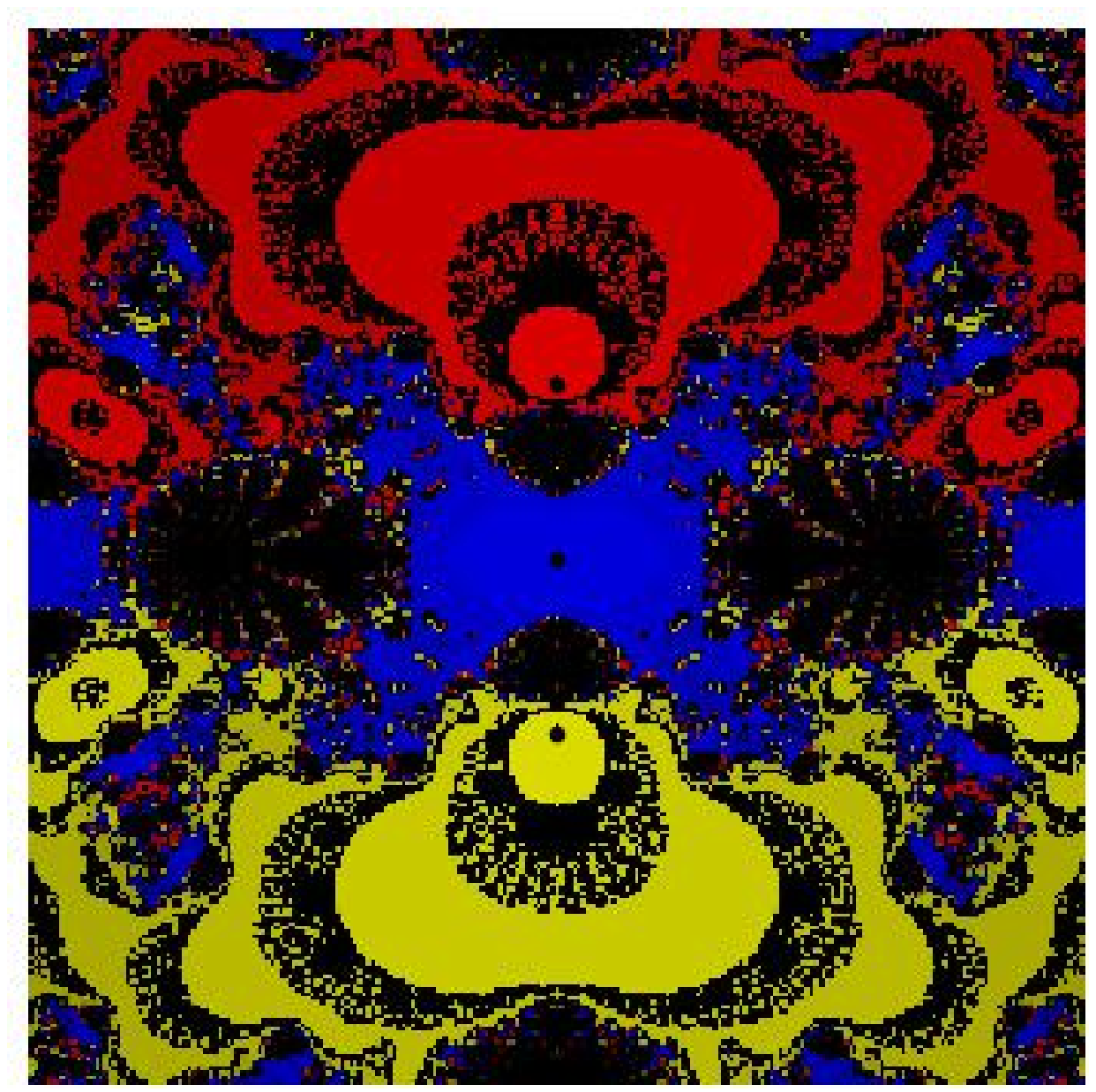}
\caption{Method (\ref{d9}) for test problem
$p_2$}\label{fig:figure10}
\end{minipage}
\end{figure}
\newpage
\begin{figure}[ht!]
\begin{minipage}[b]{0.30\linewidth}
\centering
\includegraphics[width=\textwidth]{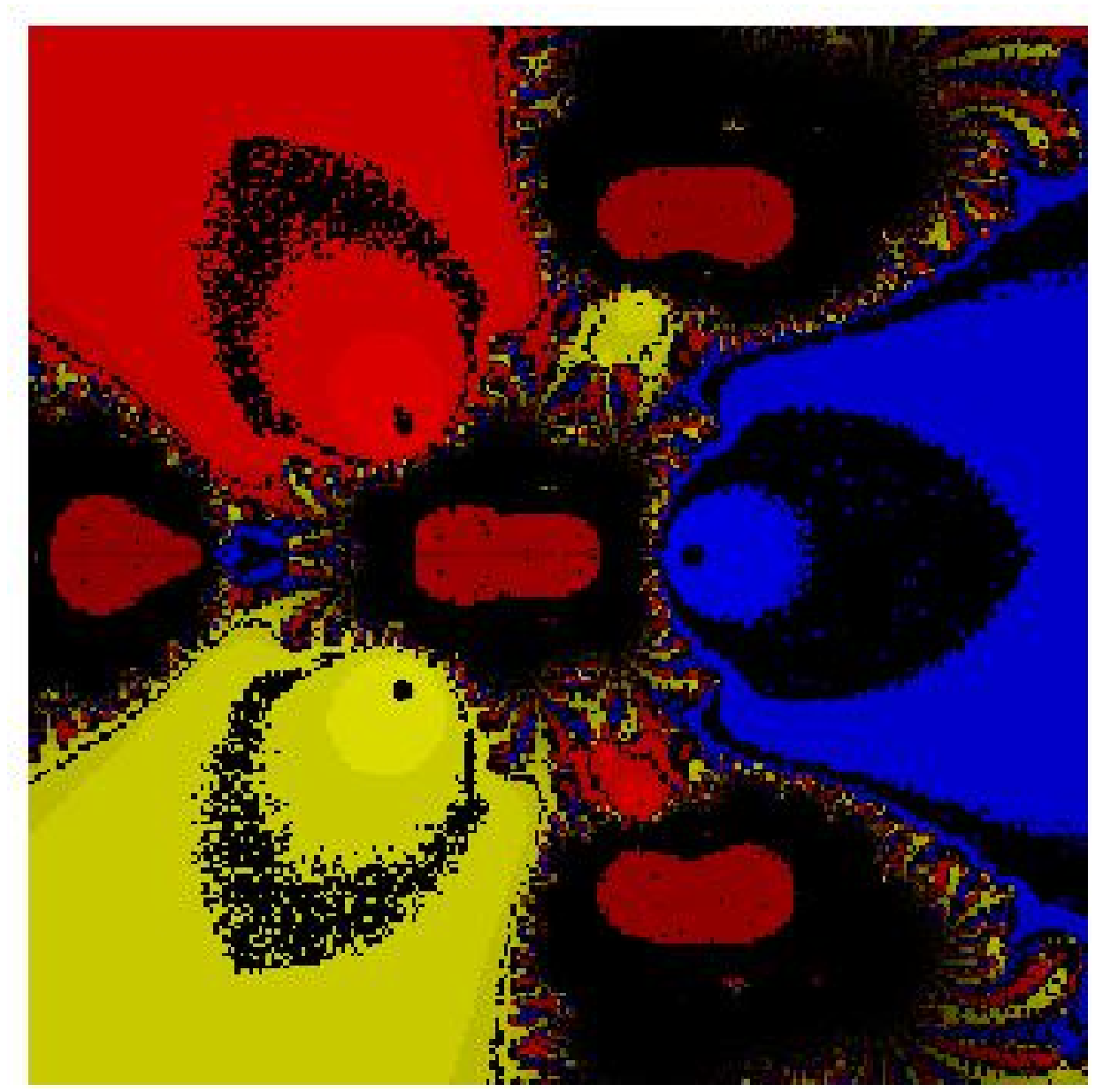}
\caption{Method (\ref{d4}) for test problem $p_3$}
\label{fig:figure11}
\end{minipage}
\hspace{0.5cm}
\begin{minipage}[b]{0.30\linewidth}
\centering
\includegraphics[width=\textwidth]{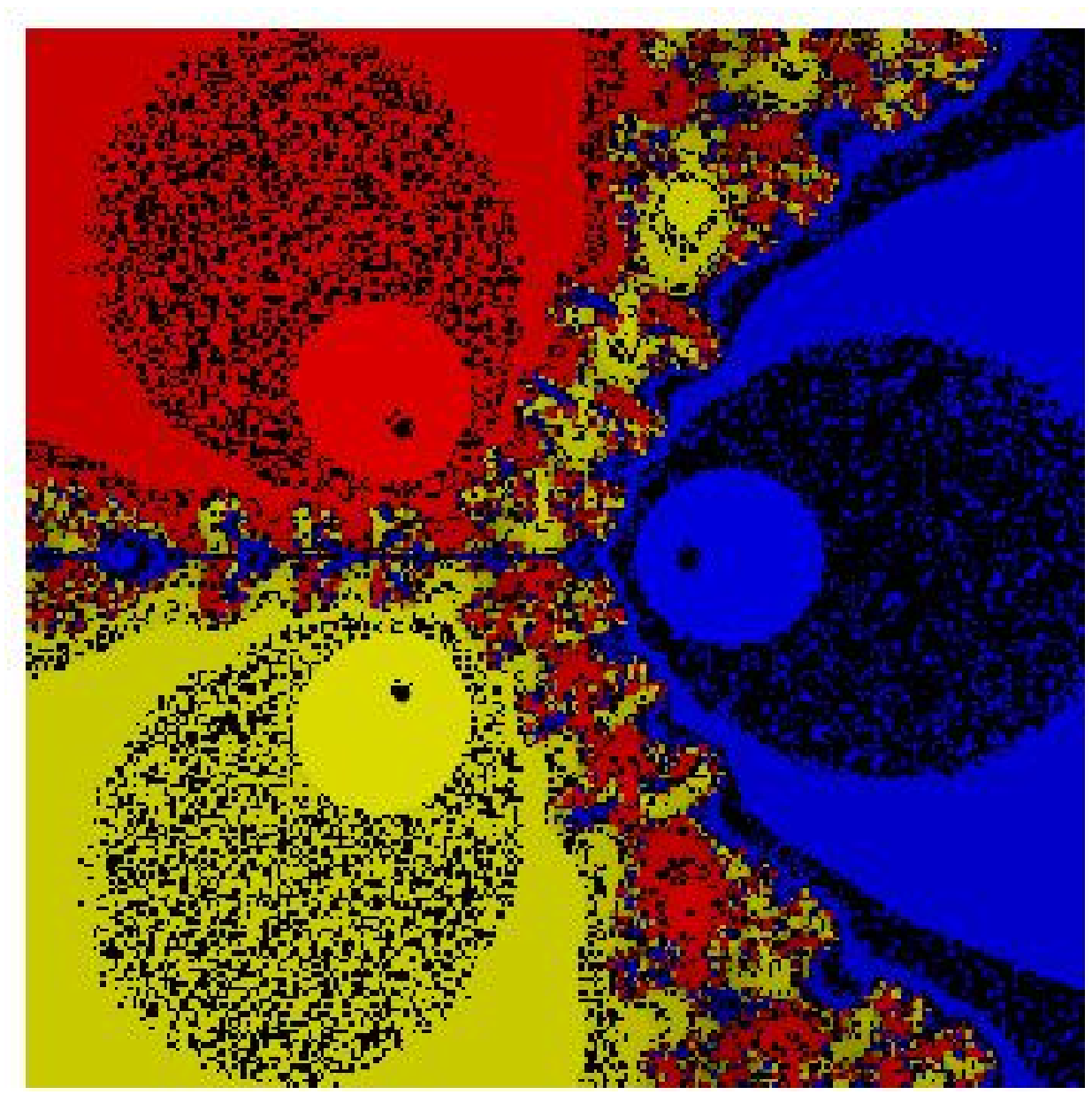}
\caption{Method (\ref{KT}) for test problem $p_3$}
\label{fig:figure12}
\end{minipage}
\hspace{0.5cm}
\begin{minipage}[b]{0.30\linewidth}
\centering
\includegraphics[width=\textwidth]{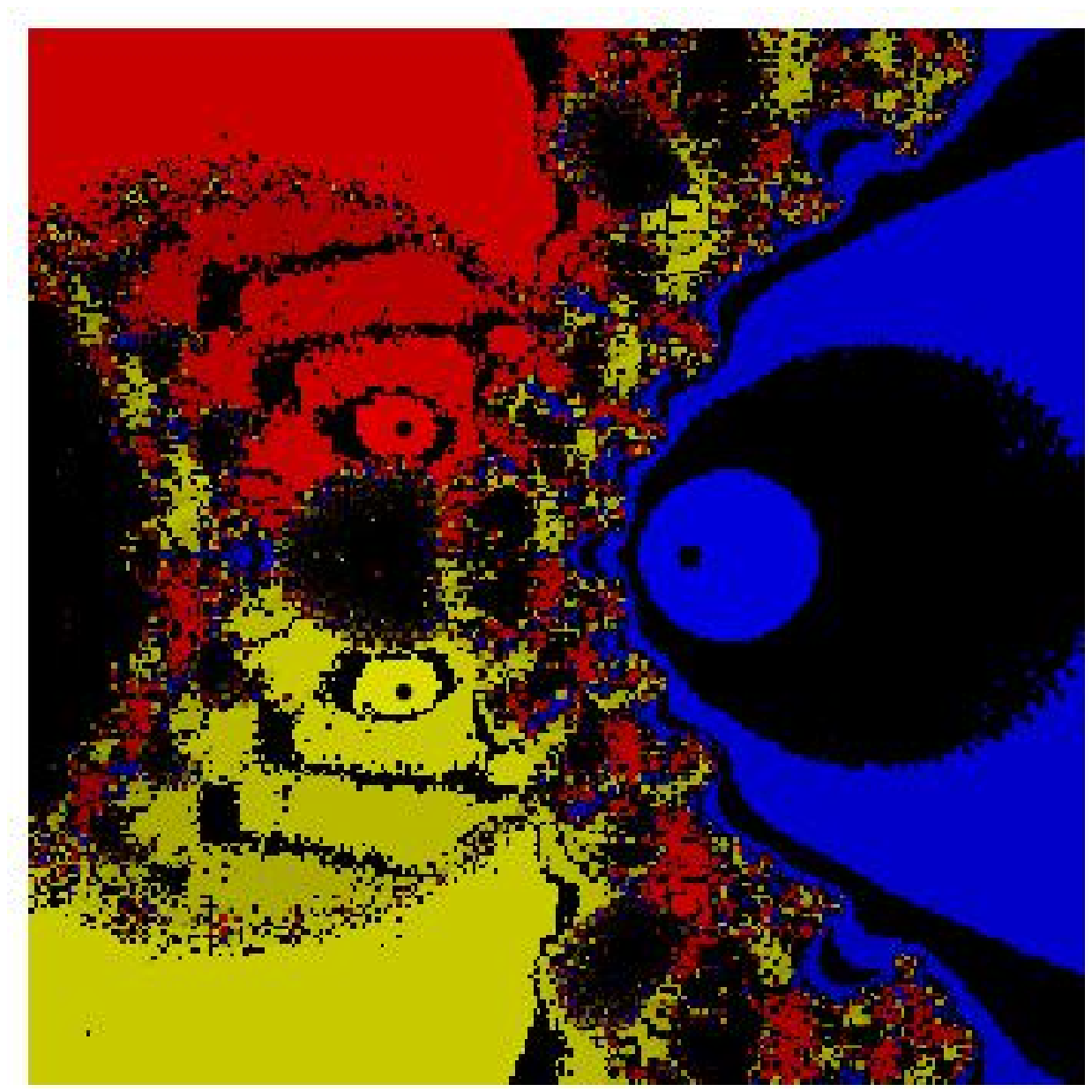}
\caption{Method (\ref{NNNN}) for test problem $p_3$}
\label{fig:figure13}
\end{minipage}
\end{figure}

\begin{figure}
\begin{minipage}[b]{0.30\linewidth}
\centering
\includegraphics[width=\textwidth]{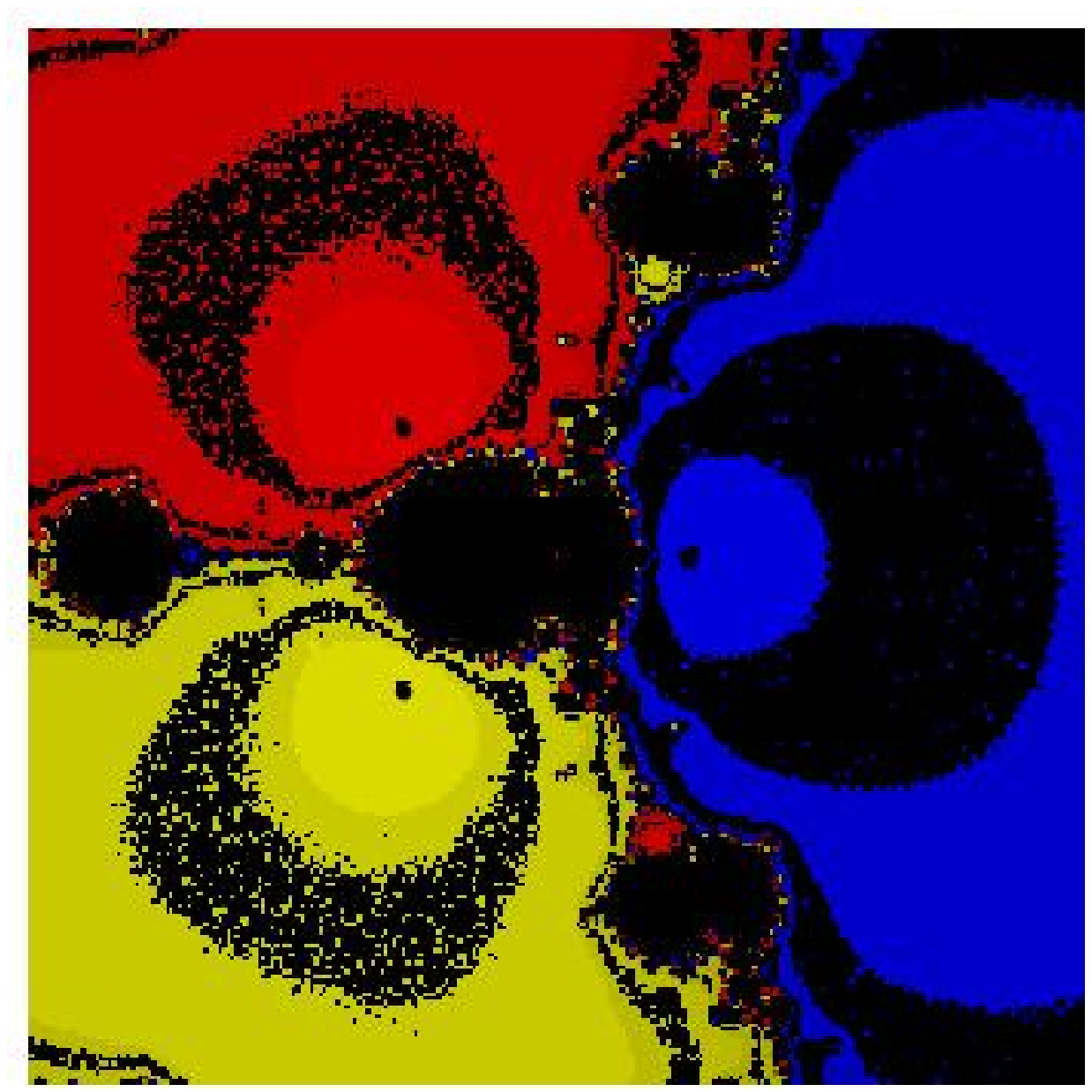}
\caption{Method (\ref{d7}) for test problem $p_3$}
\label{fig:figure14}
\end{minipage}
\hspace{1.5cm}
\begin{minipage}[b]{0.30\linewidth}
\centering
\includegraphics[width=\textwidth]{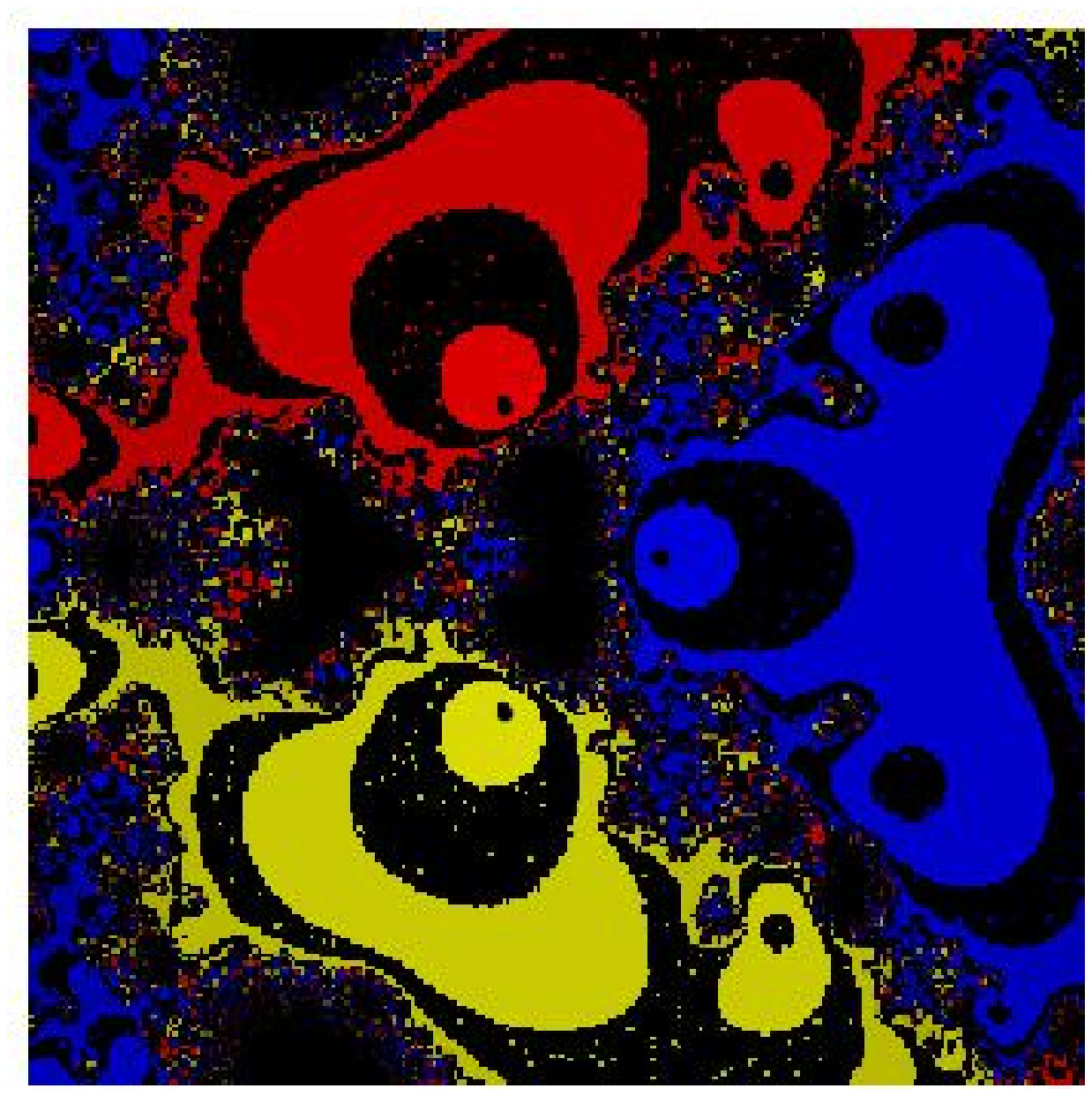}
\caption{Method (\ref{d9}) for test problem
$p_3$}\label{fig:figure15}
\end{minipage}
\end{figure}

\newpage
\section{Conclusion}
We introduce a new optimal class of four-point methods without
memory for approximating a simple root of a given nonlinear
equation. Our proposed methods use five function evaluations for
each iteration. Therefore they support Kung and Traub's conjecture.
Error and convergence analysis are carried out. Numerical examples show that
our methods work and can compete with other methods in the same class.
We used the basin of attraction for comparing the iteration
algorithms.


\begin{thebibliography}{99}

\bibitem{Amat1}Amat, S., Busquier, S., Plaza, S., Iterative root-finding methods, Unpublished report, (2004).

\bibitem{Amat2}Amat, S., Busquier, S., Plaza, S., Review of some iterative root-finding methods from a dynamical point of view, Scientia Series A: Math. Sci., 10, 3-35, (2004).

\bibitem{Amat3}Amat, S., Busquirer, S., Plaza, S., Dynamics of a family of third-order itrative methods that do not require using second derivatives, Appl. Math. Comput., 154, 735-746, (2004).

\bibitem{Amat4}Amat, S., Busquier, S., Plaza, S., Dynamics of the King and Jarratt iterations, Aeq. Math., 69, 212-223, (2005).

\bibitem{Babajee}
Babajee, D.K.R., Thukral, R., On a 4-point sixteen-oredr King
family of iterative method for solving nonlinear equations,
International Journal of Mathematics and Mathematical Sciences,
 DOI: 10.1155/2012/979245, (2012).

\bibitem{Chun1}Chun, C., Lee, M.Y., A new optimal eighth-order family of iterative methods for the solution of nonlinear equations,
Appl. Math. Comput., 223, 506-519, (2013).

\bibitem{Cordero5}
Cordero, A., Fardi, M., Ghasemi, M., Torregrosa, J.R., Accelerated
iterative methods for finding solutions of nonlinear equations and
their dynamical behavior, Calcolo, 1-14, (2012).


\bibitem{Cordero0} Cordero, A., Lotfi, T., Bakhtiari, P., Torregrosa, J.R., An efficient two-parametric family with memory for nonlinear equations,
 Numer. Algor., DOI 10.1007/s11075-014-9846-8, (2014).



\bibitem{Cordero} Cordero, A., Lotfi, T., Mahdiani, K., Torregrosa, J.R., Two optimal general classes of iterative methods
with eighth-order, Acta Appl. Math., DOI
10.1007/s10440-014-9869-0, (2013).

\bibitem{Cordero22}
Cordero, A., Torregrosa, J.R., Vassileva, M.P., Three-step
iterative methods with optimal eighth-order convergence, J.
Comput. Appl. Math., 235, 3189-3194, (2011).

\bibitem{Geum1}
Geum, Y. H., Kim, Y. I., A biparametric family optimally
convergent sixteenth-order multipoint methods with their
fourth-step weighting function as a sum of a rational and a
generic two-variable function, J. Comput. Appl. Math., 235,
3178-3188, (2011).

\bibitem{Geum2}
Geum, Y. H., Kim, Y. I., A family of optimal sixteen-order
multipoint method with a linear fraction plus a trivariate
polynomial as the fourth-step weighting function, Computers, Math.
with Appl., 61, 3278-3287, (2011).

\bibitem{Geum3}
Geum, Y. H., Kim, Y. I., A biparametric family of four-step
sixteen-oredr root-finding methods with the optimal efficiency
index, Appl. Math. Lett., 24, 1336-1342, (2011).



\bibitem{Hazrat}
 Hazrat, R., Mathematica�: A Problem-Centered Approach, Springer-Verlag, 2010.


\bibitem{Jarrat}
Jarratt, P., Some fourth order multipoint iterative methods for
solving equations, Math. Comput., 20, 434-437, (1966).

\bibitem{Khattri}
Khattri, S.K., Steihaug, T., Algorithm for forming derivative-free
optimal methods, Numer. Algor., 65, 809-824, (2014).


\bibitem{King}
King, R.F., Family of four order methods for nonlinear equations
SIAM J. Numer. Anal.,  10, 876-879, (1973).

\bibitem{Kung}
 Kung, H.T., Traub, J.F., Optimal order of one-point and multipoint iteration, Assoc. Comput. Math. 21, 643-651, (1974).

\bibitem{Lotfi0}
Lotfi, T., Cordero, A., Torregrosa, J.R., Amir Abadi, M.,
Mohammadi Zadeh, M., On generalization based on Bi et al.
iterative methods with eighth-order convergence for solving
nonlinear equations, The Scientific World Journal, Article ID
272949, (2014).

\bibitem{Lotfi}
Lotfi, T., Sharifi, S., Salimi, M., Siegmund, S., A new class of
three-point methods with optimal convergence order eight and its
dynamics, Numer. Algor., DOI 10.1007/s11075-014-9843-y, (2014).


\bibitem{Lotfi1} Lotfi, T., Tavakoli, E., On a new efficient Steffensen-like iterative class by applying a suitable self-accelerator
parameter, The Scientific World Journal, Article ID 769758,
(2014).

\bibitem{Neta0} Neta, B., On a family of multipoint
methods for nonlinear equations, Intern. J. Computer Math., 9,
353-361, (1981).

\bibitem{Neta1} Neta, B., Chun, C., Scott, M., Basin of
attractions for optimal eighth order methods to find simpl roots
of nonlinear equations, Appl. Math. Comput., 227, 567-592, (2014).

\bibitem{Neta}
Neta, B., Scott, M., Chun, C., Basin attractors for various
methods for multiple roots, Appl. Math. Comput., 218, 5043-5066,
(2012).

\bibitem{Neta2} Neta, B., Scott, M., Chun, C., Basin of
attraction for several methods to find simple roots of nonlinear
equations, Appl. Math. Comput., 218, 10548-10556, (2012).

\bibitem{Ostrowski}
Ostrowski, A.M., Solution of Equations and Systems of Equations,
Academic Press, New York, 1966.

\bibitem{Petkovic1} Petkovic, M.S., A note on the priority of optimal multipoint methods for solving nonlinear equations, Appl. Math. Comput., 219, 5249-5252, (2013).

\bibitem{Petkovic2} Petkovic, M.S., Neta, B., Petkovic, L.D., Dzunic, J., Multipoint methods for solving nonlinear equations: A survey, Appl. Math. Comput., 226, 635-660, (2014).

\bibitem{Scott}Scott, M., Neta, B., Chun, C., Basin attractors for various  methods, Appl. Math. Comput., 218, 2584-2599, (2011).


\bibitem{Sharma1}
Sharma, J.R., Sharma, R., A new family of modified Ostrowski's
methods with accelerated eighth order convergence, Numer. Algor.,
54, 445-458, (2010).

\bibitem{Soleymani2} Soleymani, F.,  Lotfi, T., Bakhtiari, P., A multi-step class of iterative methods for nonlinear
systems, Optim. Lett., 8(3), 1001-1015, (2014).

\bibitem{Stewart}Stewart, B.D., Attractor Basins of Various Root-Finding Methods M.S. thesis, Naval Postgraduate School, Department of Applied Mathematics, Monterey, CA, June, 2001.

\bibitem{Traub}
Traub, J.F., Iterative Methods for the Solution of Equations,
Prentice Hall, New York, 1964.

\bibitem{Vrscay}Vrscay, E.R., Gilbert, W.J., Extraneous fixed points, basin boundaries and chaotic dynamics for Schroder and Konig rational iteration functions, Numer. Math., 52, 1-16, (1988).

\bibitem{Fer} Weerakoon, S., Fernando, T.G.I.,
A variant of Newton's method with accelerated third-order
convergence, Appl. Math. Lett., 13, 87-93, (2000).
\end{thebibliography}
\end{document}